\documentclass[reqno,  11pt]{amsart}

\usepackage[top=3cm, bottom=2.5cm, left=2.5cm, right=2.5cm]{geometry}
\usepackage{enumerate}
\usepackage{amssymb,amsmath}
\usepackage{hyperref}
\usepackage{graphicx}
\usepackage[numbers,sort&compress]{natbib}
\newtheorem{theorem}{Theorem}

\newtheorem{lem}{Lemma}
\newtheorem{prop}{Proposition}
\theoremstyle{definition}
\newtheorem{defn}{Definition}
\theoremstyle{remark}
\newtheorem{rem}{Remark}
\newtheorem{assump}{Assumption}
\allowdisplaybreaks

\newcommand{\RR}{\mathbb{R}}

\newcommand{\OX}{\overline{X}}

\newcommand{\R}{\mathbb{R}}

\numberwithin{equation}{section} \numberwithin{lem}{section}
\numberwithin{cor}{section} \numberwithin{rem}{section}

\newenvironment{proof*}
{\pushQED{\qed}}
{\popQED\endtrivlist}

\title[]{Mathematical Analysis of the PDE Model for the Consensus-based Optimization}
\author{Jinhuan Wang$^{\,1}$, Keyu Li$^{\,1}$, and Hui Huang$^{\,2}$}
\thanks{The work of J. Wang is patially supported by National Natural Science Foundation of China, Grants No. 12171218 and Liaoning Provincial Natural Science
Foundation Program, Grants No. 2024-MS-002.}
\thanks{Corresponding author: Hui Huang}
\begin{document}
\maketitle
\begin{center}
{\footnotesize
1-School of Mathematics and Statistics, Liaoning University, Shenyang, 110036, P. R. China \\
email:  wangjh@lnu.edu.cn; lky\_lnu@163.com\\
 \smallskip
2-Department of Mathematics and Scientific Computing, Karl-Franzens-Universit\"{a}t Graz, 8010 Graz, Austria
\\
email: hui.huang1@ucalgary.ca
}
\end{center}
\maketitle
\date{}
\begin{abstract}
In this paper, we develop an analytical framework for the partial differential equation underlying the consensus-based optimization model. The main challenge arises from the nonlinear, nonlocal nature of the consensus point, coupled with a diffusion term that is both singular and degenerate. By employing a regularization procedure in combination with a compactness argument, we establish the global existence and uniqueness of weak solutions in $L^\infty(0,T;L^1\cap L^\infty(\R^d))$. Furthermore, we show that the weak solutions exhibit improved $H^2$-regularity when the initial data is regular.

\end{abstract}


\newtheorem{definition}{Definition}[section]
\newtheorem{lemma}{Lemma}[section]
\newtheorem{proposition}{Proposition}[section]
\newtheorem{corollary}{Corollary}[section]
\newtheorem{remark}{Remark}[section]
\renewcommand{\theequation}{\thesection.\arabic{equation}}
\catcode`@=11 \@addtoreset{equation}{section} \catcode`@=12

\section{Introdution}
\subsection{Background}
In this paper, we investigate the partial differential equation (PDE) that underlies the consensus-based optimization (CBO) model. CBO, introduced in \cite{carrillo2018analytical,pinnau2017consensus}, is a gradient-free optimization technique inspired by the collective behavior of interacting agents, modeled through consensus dynamics. It belongs to the class of optimization strategies known as \textit{metaheuristics} \cite{boussaid2013survey,borghi2024kinetic}, which are high-level, problem-independent algorithmic frameworks designed to find approximate solutions for complex optimization problems, particularly when exact methods are computationally infeasible. Several well-known metaheuristic approaches include evolutionary programming \cite{yao1999evolutionary}, the Metropolis-Hastings sampling algorithm \cite{chib1995understanding}, genetic algorithms \cite{holland1992genetic}, ant colony optimization (ACO) \cite{blum2005ant}, and simulated annealing (SA) \cite{bertsimas1993simulated}. Unlike traditional optimization methods that rely on gradient information, CBO utilizes swarm intelligence principles, wherein a population of particles evolves over time to reach a consensus on the optimal solution. This approach is particularly advantageous in high-dimensional, non-convex, or noisy landscapes, where gradient-based methods may struggle. By simulating interactions and communication among agents, CBO effectively balances exploration and exploitation, facilitating robust convergence to global optima. 

Like other metaheuristic methods, the versatility and simplicity of CBO have contributed to its growing adoption in solving a wide range of optimization problems across machine learning, engineering, and the sciences. These applications include, but are not limited to, global optimization on compact manifolds \cite{fornasier2021JMLR, ha2022stochastic}, handling general constraints \cite{borghi2023constrained, beddrich2024constrained}, optimizing cost functions with multiple minimizers \cite{bungert2022polarized, fornasier2025pde}, addressing multi-objective problems \cite{borghi2022consensus}, solving stochastic optimization problems \cite{bonandin2024consensus}, tackling high-dimensional machine learning challenges \cite{fornasier2022anisotropic, carrillo2021consensus}, and sampling from distributions \cite{carrillo2022consensus}. The CBO framework has been further enhanced through the incorporation of momentum \cite{chen2022consensus}, memory effects \cite{huang2024self, totzeck2020consensus,riedl2024leveraging}, second-order algorithms \cite{byeon2025consensus, cipriani2022zero, huang2023global}, optimal control strategies \cite{huang2024fast}, mirror descent \cite{bungert2025mirrorcbo}, random batch interactions \cite{ko2022convergence}, truncated noises \cite{fornasier2025consensus}, smoothing effects \cite{wei2025smoothing} and jump processes \cite{kalise2023consensus}. Additionally, CBO has been successfully applied to multi-player games \cite{chenchene2023consensus}, min-max problems \cite{huang2024consensus, borghi2024particle}, multi-level optimization \cite{herty2024multiscale,trillos2024cb}, and clustered federated learning \cite{carrillo2023fedcbo}.

Let us consider  the following optimization problem
\begin{equation*}
	\text{Find }\; x_* \in \arg\min_{x\in \RR^d} f(x),\,
\end{equation*}
where $f$ can be a non-convex non-smooth objective function that one wishes to minimize. Then the CBO dynamic considers the following system of $N$ interacting particles, denoted as $\{X_{\cdot}^i\}_{i=1}^N$, which satisfies
\begin{align}\label{CBO: particle}
   dX_t^i = -\lambda\big(X_t^i - m_\alpha^f(\rho_t^N)\big)dt + \sigma |X_t^i - m_\alpha^f(\rho_t^N)|dB_t^i\qquad i=1,\cdots, N=:[N],
\end{align}
where $\lambda,\sigma>0$ are real constants, $\rho_t^N:=\frac{1}{N}\delta_{X_t^i}$ is the empirical measure associated to particle system \eqref{CBO: particle}, and the current consensus point $ m_\alpha^f(\rho_t^N)$ satisfies
\begin{align}
    m_\alpha^f(\rho_t^N)
    :=\frac{\int_{\R^d}xe^{-\alpha f(x)}d\rho_t^N(x)}{\int_{\R^d}e^{-\alpha f(x)}d\rho_t^N(x)}
    =\frac{\frac{1}{N}\Sigma_{i=1}^{N}X_t^ie^{-\alpha f(X_t^i)}}{\frac{1}{N}\Sigma_{i=1}^{N}e^{-\alpha f(X_t^i)}}\,,
\end{align}
which is a convex combination of particle positions weighed according to the weight function $w_\alpha^f(x):=e^{-\alpha f(x)}$.   This choice of weight function $\omega_\alpha^f(\cdot)$ is motivated by the well-known Laplace's principle \cite{miller2006applied,MR2571413}, see also \cite[Appendix A.2]{huang2024consensus}, which formally implies that
\begin{equation*}
	 m_\alpha^f(\rho_t^N)\approx \arg\min\{f(X_t^1),\dots,f(X_t^N)\}
\end{equation*}
for large values of $\alpha$.
Moreover, the particle system is initialized with independent and identically distributed (i.i.d.) data $\{X_{0}^i\}_{i=1}^N$, where each $X_0^i$ is distributed according to a given measure $\rho_0 \in \mathcal{P}(\RR^d)$ for all $i \in [N]$.  Notably, the well-posedness of the particle system \eqref{CBO: particle} has been showned in \cite[Theorem 2.1]{carrillo2018analytical} when $\rho_0\in \mathcal{P}_2(\RR^d)$.

When the consensus is reached $(t=\infty )$, one expects the following
\begin{equation*}
	\frac{1}{N}\sum_{i=1}^{N}X_\infty^i\approx m_\alpha^f(\rho_\infty^N)\approx x_*
\end{equation*}
holds for sufficient large $\alpha$, thus the global minimizer $x_*$ can be founded.  In the following Figure \ref{fig:Rastrigin}, using standard Euler–Maruyama scheme, we successfully apply CBO particle system \eqref{CBO: particle}  to the Rastrigin function 
\[
R(x) := 10+(x-1)^2-10\cos(2\pi(x-1))\,,
\]
which is a popular benchmark nonconvex function in optimization with multiple local minimizers and a unique global minimizer.
\begin{figure}[h]
	\centering
	\includegraphics[width=0.63\textwidth]{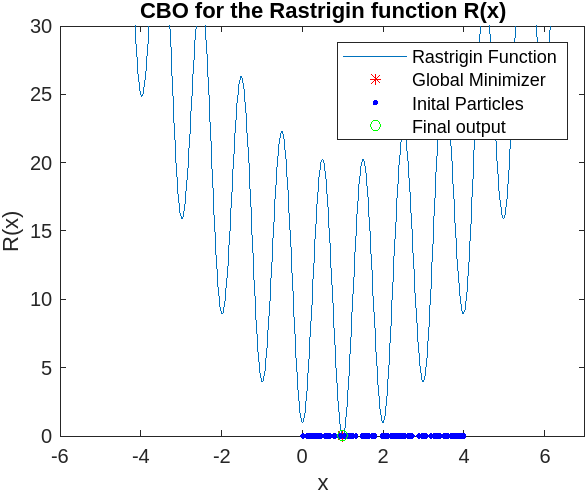}
	\caption{We apply the CBO particle system \eqref{CBO: particle} to the Rastrigin function $R(x)$, which has a unique global minimizer $x_*=1$ (the red star). The initial particles (the blue dots) are sampled uniformly in $[0,4]$. The simulation parameters are $N=100,\lambda=1,\sigma=1,\alpha=10^{15}, d t=0.01$ and $T=100$. The final output is $m_\alpha^f(\rho_T^N)$ (the green circle).
	}
	\label{fig:Rastrigin}
\end{figure}

\subsection{Motivation and Contribution}
The convergence proof \cite{carrillo2018analytical,fornasier2024consensus,huang2025faithful} of the CBO method is typically conducted in the context of large-particle limit \cite{gerber2023mean,huang2022mean,huang2024uniform,koss2024mean,fornasier2020consensus}. Specifically, rather than analyzing the $N$--particle system \eqref{CBO: particle} directly, one considers the limit as $N$ approaches infinity and examines the corresponding McKean-Vlasov process $\OX_{\cdot}$, which is governed by the following equation
\begin{align}\label{CBO: mfq}
   d\overline{X}_t = -\lambda\big(\overline{X}_t - m_\alpha^f(\rho_t)\big)dt + \sigma |\overline{X}_t - m_\alpha^f(\rho_t)|dB_t,
\end{align}
where
\begin{align*}
    m_\alpha^f(\rho_t)
    =\frac{\int_{\R^d}xe^{-\alpha f(x)}d\rho_t(x)}{\int_{\R^d}e^{-\alpha f(x)}d\rho_t(x)}\quad \mbox{with }\rho_t=\mbox{Law}(\OX_t)\,.
\end{align*}
Applying It\^{o}'s formula one can easily see that $\rho$ satisfies the following nonlinear Fokker-Planck equation:
\begin{align}\label{pde}
\begin{cases}
    &\partial_t\rho = \lambda\nabla\cdot\big((x-m_\alpha^f(\rho_t))\rho\big) + \frac{\sigma^2}{2}\Delta\big(|x-m_\alpha^f(\rho_t)|^2\rho\big),\\
	&\rho(0,x)=\rho_0(x)\,.
\end{cases}	
\end{align}

Notably, the well-posedness of \eqref{CBO: mfq} was established in \cite[Theorem 3.1, Theorem 3.2]{carrillo2018analytical} under the assumption that $0\le\rho_0 \in \mathcal{P}_4(\mathbb{R}^d)$, ensuring the existence of a measure solution $\rho \in \mathcal{C}([0,T],\mathcal{P}_2(\mathbb{R}^d))$ to the PDE \eqref{pde}. However, proving the global convergence of CBO methods may require more regular solutions; see, for instance, \cite{fornasier2021JMLR,trillos2024cb,fornasier2025regularity}.

 This motivates the present work, which aims to develop an analytical framework for the PDE \eqref{pde} with the objective of obtaining regular solutions. 
The main difficulty arises from the consensus term $m_\alpha^f(\rho_t)$, which is highly nonlinear and nonlocal. Moreover, the diffusion term $\Delta\big(|x - m_\alpha^f(\rho_t)|^2 \rho\big)$ presents significant challenges: it becomes degenerate when $x = m_\alpha^f(\rho_t)$ and singular as $x \to \infty$. Consequently, standard PDE results do not directly apply in this setting. To overcome these difficulties we can consider the following regularization problem
for $R,\varepsilon>0$
\begin{align}\label{pdesR1}
\begin{cases}
     & \partial_t\rho_{\varepsilon,R} = \lambda\nabla\cdot\Big(\big(x- m_\alpha^f(\rho_{\varepsilon})\big)\phi^R\big(x- m_\alpha^f(\rho_{\varepsilon})\big)\rho_{\varepsilon,R}\Big)\\ 
   &\qquad\qquad\qquad+ \frac{\sigma^2}{2}\Delta\big(\big(\big|(x- m_\alpha^f(\rho_{\varepsilon}))\phi^R(x- m_\alpha^f(\rho_{\varepsilon}))\big|^2+\varepsilon^2\big)\rho_{\varepsilon,R}\big),\\
  & \rho_{\varepsilon,R}(0,x)=\rho_{0\varepsilon}(x):=\rho_{0}(x)*J_{\varepsilon}(x),
\end{cases}
\end{align}
where $J_{\varepsilon}(x)=\frac{1}{\varepsilon^d}J(\frac{x}{\varepsilon})$, $J(x)$ is the standard mollifier, then $\rho_{0\varepsilon}$ is nonnegative and it is easy to check that  $\rho_{0\varepsilon}$ satisfies the following inequalities that for any $p\in[1,\infty]$,
\begin{align}
&\|\rho_{0\varepsilon}\|_{L^p{(\R^d)}}\le \|\rho_0\|_{L^p{(\R^d)}},\quad\int_{\R^d}|x|^4\rho_{0\varepsilon}dx\le C\Big(\int_{\R^d}|x|^4\rho_{0}dx + \varepsilon^4\Big).
\end{align}
Moreover, here we take $\phi^R(v):= \phi(|v|/R),\ v\in\mathbb{R}^d$ with $\phi(v)\in C^2[0,+\infty)$ satisfying
\begin{eqnarray*}
\phi(v) =
\begin{cases}
1,& v\le 1,\\
0,& v\ge 2,
\end{cases}
\quad 0\le\phi\le1,\ |\phi'|\le C\ \mbox{and}\ |\phi''|\le C.
\end{eqnarray*}
Then for all $i,j\in[d]$ we have 
\begin{align}\label{vphi}
&\|\phi^R(v)\|_{L^\infty(\R^d)} \le 1,\quad\|\partial_{v_i}\phi^R(v)\|_{L^\infty(\R^d)} \le \frac{C}{R}\ \mathrm{and}\ \|\partial_{v_i,v_j}^2\phi^R(v)\|_{L^\infty(\R^d)} \le \frac{C}{R^2}.
\end{align}
Now for any fixed $R,\varepsilon>0$ the new regularization PDE \eqref{pdesR1} is nondegenerate because of $\varepsilon$ and nonsingular because of $\phi^R$, and the classical parabolic equation theory implies that the above regularization problem has a global smooth non-negative solution $\rho_{\varepsilon,R}$ for $t>0$ if the initial data are non-negative.  

The main goal of this paper is to prove that there exists a unique solution $\rho$ to PDE \eqref{pde} in the sense of the following definition. 
\setcounter{theorem}{0}
\begin{defn}\label{def}
(Definition of weak solutions). For any fixed $T>0$, $\rho$ is called a weak solution to the problem \eqref{pde} on $[0,T]$ if $\rho$
satisfies
\begin{enumerate}[(i)]
\item regularities
\begin{align*}
&\partial_t\rho \in L^2(0,T;W_{loc}^{-2,\frac{4}{3}}(\R^d)), \quad |x|^4\rho\in L^\infty(0,T;L^1(\R^d)),\\
&\rho\in L^\infty(0,T;L^1\cap L^2(\RR^d))\cap L^2(0,T;W_{loc}^{1, \gamma}(\RR^d)),~~\gamma\in(1,\frac{2d}{2+d});
\end{align*}
\item for any $\varphi\in C_c^{\infty}((0,T)\times \mathbb{R}^d)$, the following integration equality holds
\begin{align*}
\int_0^T\int_{\R^d}\partial_t\rho\varphi+\lambda\Big(\big(x-m_\alpha^f(\rho)\big)\rho\Big)\cdot\nabla\varphi+\frac{\sigma^2}{2}\nabla\Big(\big|x-m_\alpha^f(\rho)\big|^2\rho\Big)\cdot\nabla\varphi dxdt =0.
\end{align*}
\end{enumerate}
\end{defn}
During the preparation of this work, we became aware of the recent study~\cite{fornasier2025regularity}, which also investigates the well-posedness of the CBO PDE \eqref{pde}. In particular, Theorem~2.5 in~\cite{fornasier2025regularity} establishes the existence of regular solutions for a class of general linear PDEs~\cite[(11)--(12)]{fornasier2025regularity} using a Galerkin approximation. By combining this result with the existence of measure-valued solutions to the original nonlinear CBO PDE \eqref{pde}, as shown in~\cite[Theorem 3.2]{carrillo2018analytical}, the authors derive the existence of regular solutions for the nonlinear problem \cite[Section 2.4]{fornasier2025regularity}. 
Our analysis, developed independently and concurrently, adopts a different methodological approach. In contrast to~\cite{fornasier2025regularity}, we establish the existence of regular solutions via a compactness argument applied to a regularization version of the CBO PDE, without relying on the intermediate step of constructing measure-valued solutions. Our approach offers a direct and self-contained framework that may be more accessible to researchers in the PDE community.

The result on existence and uniqueness of solutions is as follows.
\begin{theorem}\label{thm1}
Let $f$ satisfy Assumption \ref{ass}, and the non-negative initial data $\rho_0\in L^1\cap L^2(\R^d)\cap \mathcal{P}_4(\R^d)$. Then, for any $T>0$, there exists a unique weak solution to the model \eqref{pde} in the sense of Definition \ref{def}.      
\end{theorem}

Moreover, we obtain the boundedness and $H^2$-regularity of solutions.
\begin{theorem}\label{thm2}
Let $f\in\mathcal{C}(\R^d)$ satisfy Assumption \ref{ass}. Then the solution $\rho$ to problem \eqref{pde} satisfies 
    \begin{itemize}
        \item[(i)] if the non-negative initial data $\rho_0\in L^1\cap L^\infty(\R^d)\cap \mathcal{P}_4(\R^d)$, then
	   \begin{equation*}
	   \rho\in L^\infty(0,T;L^p(\R^d)),\quad p\in[2,\infty]; 
	   \end{equation*}
	   \item[(ii)] if the non-negative initial data $\rho_0\in L^1\cap H^2(\R^d)\cap \mathcal{P}_4(\R^d)$, then
	   \begin{equation*}
	   \rho\in L^\infty(0,T;H^2(\R^d)),\quad \partial_t\rho\in L^\infty(0,T;L^2_{loc}(\mathbb{R}^d)).
	   \end{equation*}
	\end{itemize}
\end{theorem}


This paper is organized as follows. Section 2 is devoted to establishing uniform estimates for solutions of the regularization problem \eqref{pdesR}. In Section 3, we prove the existence of solutions to the original problem \eqref{pdesR1} through compactness methods, while the uniqueness follows from the well-posedness of the associated stochastic differential equation \eqref{CBO: mfq}. Finally, Section 4 presents the boundedness and $H^2$-regularity results for the obtained solutions.

\section{The uniform estimate of solutions to the regularization problem}
In this section, we show the uniform estimates of solutions to the regularization problem \eqref{pdesR1}. Before proving them, we outline the essential assumptions and foundational lemmas required for the subsequent analysis.

We first state the assumptions imposed on the cost function $f$.
\begin{assump}\label{ass}
Throughout we are interested in objective function $f\in\mathcal{C}(\R^d)$, for which
	\begin{enumerate}
		\item $f:~\RR^d\to \RR$ is bounded from below by $\underline f=\min f$, and there exist $L_{f}>0, s\geq 0$ constants such that
		\begin{equation*}
			|f(x)-f(y)|\leq L_f(1+|x|+|y|)^s|x-y|, \quad \forall x,y\in \RR^d;
		\end{equation*}
		\item There exist constants $c_\ell,c_u,c_0>0$  and $\ell>0$ such that
		\begin{equation*}
			c_\ell(|x|^\ell-c_0)	\leq f-\underline f\leq c_u(|x|^\ell+1),\quad \forall x\in \RR^d.
		\end{equation*}
	\end{enumerate}
\end{assump}

In what follows, we denote by $W^{k,p}(\RR^d)$ the Sobolev space  and by $W_{loc}^{k,p}(\RR^d)$ its local version,
and $|\cdot|$ is the standard Euclidean norm in $\RR^d$; $\mathcal{P}(\RR^d)$ denotes the space of probability measures on $\RR^d$, and $\mathcal{P}_p(\RR^d)$ with $p\geq 1$ contains all $\mu\in \mathcal{P}(\RR^d)$ such that $\mu(|\cdot|^p):=\int_{\RR^d}|x|^pd\mu(x)<\infty$; it is equipped with $p$-Wasserstein distance $W_p(\cdot,\cdot)$. 
Lastly, we define $\mathcal{P}_{p,K}(\mathbb{R}^{d}) := \{\mu \in \mathcal{P}_{p}(\mathbb{R}^{d}):\, \mu(|\cdot|^{p})\leq K\}$. Moreover when $\mu\in \mathcal{P}(\RR^d)$ has a density we abuse the notion $d\mu(x)=\mu(x)dx$.

Now, let us recall some estimates on $m_\alpha^f(\mu)$ from \cite[Proposition 3.1, Proposition A.3]{gerber2023mean}.
\begin{lem}
\label{lem: useful estimates}
Suppose that $f$ satisfies Assumption \ref{ass}. Then for all $K>0$, 
there exists some constant $L_{\mathfrak{m}}>0$ depending only on $L_f,K,\alpha$ such that
\begin{equation}\label{lemeq1}
    |m_\alpha^f(\mu)-m_\alpha^f(\nu)|\leq L_{m} W_2(\mu,\nu)\quad \forall (\mu,\nu)\in \mathcal{P}_{2,K}(\RR^d)\times \mathcal{P}_{2,K}(\RR^d)\,.
\end{equation}
Moreover for all $ p\geq 1$, there exists constant $C_m>0$ depending only on $p,c_\ell,c_u,c_0,\ell,\alpha$ such  that
\begin{equation}\label{ms}
    |m_\alpha^f(\mu)|^p\leq C_m\int_{\RR^d}|x|^pd\mu(x)\quad \forall \mu\in \mathcal{P}_p(\RR^d)\,.
\end{equation}
\end{lem}

For the purpose of convenience, let us define $$h^R(x):=(x- m_\alpha^f(\rho_{\varepsilon}))\phi^R(x- m_\alpha^f(\rho_{\varepsilon})).$$
Then the equation \eqref{pdesR1} can be rewritten as
\begin{align}\label{pdesR}
\begin{cases}
     & \partial_t\rho_{\varepsilon,R} = \lambda\nabla\cdot\big(h^R(x)\rho_{\varepsilon,R}\big) + \frac{\sigma^2}{2}\Delta\big((|h^R(x)|^2+\varepsilon^2)\rho_{\varepsilon,R}\big),\\
  & \rho_{\varepsilon,R}(0,x)=\rho_{0\varepsilon}(x):=\rho_{0}(x)*J_{\varepsilon}(x),  
\end{cases}
\end{align}
where $h^R = (h^1, h^2,\cdots,h^d)$ has the following properties for all $i,j,k\in[d]$
\begin{align}\label{hR}
\| h^i\|_{L^\infty(\R^d)}\le CR,\qquad\|\partial_{x_j}h^i\|_{L^\infty(\R^d)}\le C \ \mathrm{and}\ \|\partial_{x_j,x_k}^2h^i\|_{L^\infty(\R^d)}\le \frac{C}{R}.
\end{align}
The classical parabolic theory implies that the above regularization problem has a global smooth solution $\rho_{\varepsilon,R}$ for $t > 0$. We prove the non-negativity if the initial data is non-negative. 
\begin{lem}\textup{(Non-negativity)} 
The solution $\rho_{\varepsilon,R}$ to \eqref{pdesR} is non-negative.
\end{lem}
\begin{proof}
Define $\rho_-:=\min{\{0,\rho_{\varepsilon,R}\}}$ and we consider the following space for any $t\in[0,T]$
$$\mathcal{A}_t={\{x:\rho_{\varepsilon,R}(t,x)<0\}}~~\text{and}~~\partial\mathcal{A}_t={\{x:\rho_{\varepsilon,R}(0,x)=0\}}.$$
For the initial $\rho_0\ge 0$, we find $\rho_-(0,x) = 0$. Multiply $\rho_-$ on both sides of the equation \eqref{pdesR} and integrate over $\R^d$. Then, for all $0<t<T$, we have
\begin{align*}
\int_{\R^d}\rho_-\partial_t\rho_{\varepsilon,R} dx 
=& \lambda\int_{\R^d}\rho_- \nabla \cdot( h^R(x)\rho_{\varepsilon,R}) dx +\frac{\sigma^2}{2}\int_{\R^d}\rho_-\Delta\big((|h^R(x)|^2+\varepsilon^2)\rho_{\varepsilon,R}\big)dx\\
=& -\lambda\int_{\mathcal{A}_t}\nabla\rho_-\cdot h^R(x)\rho_- dx -\frac{\sigma^2}{2}\int_{\mathcal{A}_t}\nabla\rho_-\cdot\nabla\big(|h^R(x)|^2\rho_-\big)dx -\frac{\sigma^2\varepsilon^2}{2}\int_{\mathcal{A}_t}|\nabla\rho_-|^2dx\\
=&-\frac{1}{2}(\lambda + \sigma^2)\int_{\mathcal{A}_t}h^R(x)\cdot\nabla(\rho_-^2)dx -\frac{\sigma^2}{2}\int_{\mathcal{A}_t}|\nabla\rho_-|^2|h^R(x)|^2dx-\frac{\sigma^2\varepsilon^2}{2}\int_{\mathcal{A}_t}|\nabla\rho_-|^2dx\\
\le&\frac{d}{2}(\lambda + \sigma^2)\int_{\mathcal{A}_t}\rho_-^2dx.
\end{align*}
By Gr\"onwall's inequality, we have for a fixed $T$ and any $t\in[0,T]$
$$\|\rho_-\|_{L^2(\R^d)} \le \exp{\{\frac{d}{2}(\lambda + \sigma^2)T\}}\|\rho_-(0,x)\|_{L^2(\R^d)} = 0.$$
Thus for any $t\in [0,T]$, we obtain $\rho_{\varepsilon,R}\ge0$ a.e. in $\R^d$.
\end{proof}
Furthermore, the mass is conserved in the following sense
\begin{align}\label{mc}
\int_{\R^d}\rho_{\varepsilon,R}(t,x)dx = \int_{\R^d}\rho(0,x)dx \quad \forall t>0\,.
\end{align}

In the following, we present several preliminary lemmas. We begin by estimating the uniform $\| \cdot \|_{L^{\infty}(0,T;L^2(\mathbb{R}^d))}$ norm of the regularization solution.
\begin{lem}\label{lem2.1}
Assume that $f$ satisfies Assumption \ref{ass}, and let the non-negative initial data $\rho_0\in L^1\cap L^2(\R^d)$ and $\rho_{\varepsilon,R}$ be the solution to the regularization problem \eqref{pdesR}. Then, for any fixed $T>0$, $\rho_{\varepsilon,R}$ satisfies the following uniform estimates
\begin{align}\label{rhoR}
\|\rho_{\varepsilon,R}\|_{L^{\infty}(0,T;L^2(\R^d))}\leq C,\quad 
\sigma^2\int_0^T\int_{\R^d}|h^R(x) |^2|\nabla\rho_{\varepsilon,R}|^2dxdt\leq C,
\end{align}
where $C$ is independent of $R,\varepsilon$.
\end{lem}
\begin{proof}
Multiply $\rho_{\varepsilon,R}$ to the both hand in the regularization equation \eqref{pdesR} to get that for fixed $T>0$, $0\leq t\leq T$ it holds that
\begin{align}\label{L12}
&\frac{1}{2}\frac{d}{dt}\|\rho_{\varepsilon,R}\|_{L^2(\R^d)}^2 \nonumber\\
=&-\lambda\int_{\R^d}\nabla\rho_{\varepsilon,R}\cdot h^R(x)\rho_{\varepsilon,R} dx - \frac{\sigma^2}{2}\int_{\R^d}\nabla\rho_{\varepsilon,R}\cdot\nabla\Big(\big(|h^R(x) |^2+\varepsilon^2\big)\rho_{\varepsilon,R}\Big)dx\nonumber\\
=& -\frac{\lambda}{2}\int_{\R^d}\nabla(\rho_{\varepsilon,R})^2\cdot h^R(x)dx - \frac{\sigma^2}{2}\int_{\R^d}\nabla\rho_{\varepsilon,R}\cdot\nabla\big(|h^R(x) |^2\rho_{\varepsilon,R}\big)dx - \frac{\varepsilon^2\sigma^2}{2}\int_{\R^d}|\nabla\rho_{\varepsilon,R}|^2dx \nonumber\\
\le& -\frac{\lambda}{2}\int_{\R^d}\nabla(\rho_{\varepsilon,R})^2\cdot h^R(x)dx 
-\frac{\sigma^2}{4}\int_{\R^d}\nabla(\rho_{\varepsilon,R})^2\cdot \nabla (|h^R(x)|^2)dx 
- \frac{\sigma^2}{2}\int_{\R^d}|\nabla(\rho_{\varepsilon,R})|^2|h^R(x) |^2dx\nonumber\\
=& \frac{\lambda}{2}\int_{\R^d}\rho_{\varepsilon,R}^2 |\nabla\cdot h^R(x)|dx 
+\frac{\sigma^2}{4}\int_{\R^d}\rho_{\varepsilon,R}^2\Delta
(|h^R(x)|^2)dx 
- \frac{\sigma^2}{2}\int_{\R^d}|\nabla\rho_{\varepsilon,R}|^2|h^R(x)|^2dx.
\end{align}
By \eqref{hR}, we have
\begin{align*}
&|\nabla\cdot h^R|\le\sum\limits_{j=1}^d|\partial_{x_j} h^j|\le C,\\
&|\Delta(|h^R|^2)|\le2\sum\limits_{j=1}^d|\partial_{x_j} h^j|^2 + 2 \sum\limits_{j=1}^d|h^j\partial_{x_j,x_j}^2 h^j|\le C,
\end{align*}
where $C$ is a constant independent of $R,\varepsilon$. So \eqref{L12} becomes 
\begin{align}\label{L2}
\frac{d}{dt}\|\rho_{\varepsilon,R}\|_{L^2(\R^d)}^2
\le&- \sigma^2\int_{\R^d}|h^R(x) |^2|\nabla\rho_{\varepsilon,R}|^2dx+C(\lambda+\sigma^2)\|\rho_{\varepsilon,R}\|_{L^2(\R^d)}^2. 
 \end{align}
Hence Grönwall's inequality implies $\rho_{\varepsilon,R}$ satisfies for any fixed $T>0$ and all $t\in[0,T]$,
\begin{align*}
\|\rho_{\varepsilon,R}\|_{L^2(\R^d)}\leq\exp {\{\frac{C}{2}(\lambda+\sigma^2)T\}}\|\rho_{0}\|_{L^2(\R^d)}.
\end{align*}
Thus, we obtain
$$\|\rho_{\varepsilon,R}\|_{L^{\infty}(0,T;L^2(\R^d))}\leq C,$$
where $C$ is independent of $R, \varepsilon$. Integrating on both sides of \eqref{L2} over the time $[0,T]$, we obtain
\begin{align}\label{L2T2}
\sigma^2\int_0^T\int_{\R^d}|h^R(x) |^2|\nabla\rho_{\varepsilon,R}|^2dxdt &\le C(\lambda,\sigma)\int_0^T\|\rho_{\varepsilon,R}\|_{L^2(\R^d)}^2dt - \big(\|\rho_{\varepsilon,R}\|_{L^2(\R^d)}^2 -\|\rho_0\|_{L^2(\R^d)}^2\big)\nonumber\\
&\le  C(\lambda,\sigma)\int_0^T\|\rho_{\varepsilon,R}\|_{L^2(\R^d)}^2dt + \|\rho_0\|_{L^2(\R^d)}^2\nonumber\\
&\le C(T,\lambda,\sigma,d).
\end{align}
Hence \eqref{rhoR} holds true. The demonstration of this lemma is hereby concluded.
\end{proof}

Next, we establish the boundedness of the fourth moment, a crucial prerequisite for proving Lemma~\ref{lem3.6}.
\begin{lem}\label{lem3.2}\textup{(Boundedness of moments)}
Assume that $f$ satisfies Assumption \ref{ass}. Let the non-negative 
 initial data $\rho_0\in L^1\cap L^2(\R^d)\cap \mathcal{P}_4(\RR^d)$. Then, there is a constant $M_4>0$ independent of $R,\varepsilon$ such that
\begin{align}\label{4m}
\sup_{t\in[0,T]}\int_{\R^d}|x|^4\rho_{\varepsilon,R} dx\le M_4.    
\end{align}
\end{lem}
\begin{proof}
First, we prove the boundedness of the second moment. Multiplying \eqref{pdesR} by $|x|^2$, integrating the resulting equation and using the integration by parts, we have that for any $t\in[0,T]$
\begin{align}\label{2m2}
\frac{d}{dt}\int_{\R^d}|x|^2\rho_{\varepsilon,R} dx 
=& \lambda\int_{\R^d}|x|^2\nabla\cdot\big(h^R(x)\rho_{\varepsilon,R}\big)dx + \frac{\sigma^2}{2}\int_{\R^d}|x|^2\Delta\Big(\big(|h^R(x)|^2+\varepsilon^2\big)\rho_{\varepsilon,R}\Big)dx\nonumber\\
=& -2\lambda\int_{\R^d}x\cdot h^R(x)\rho_{\varepsilon,R} dx + \sigma^2d\int_{\R^d}\big(|h^R(x)|^2 +\varepsilon^2\big)\rho_{\varepsilon,R} dx.
\end{align}
Notice $h^R=(x-m_\alpha^f(\rho_{\varepsilon,R}))\phi^R(x-m_\alpha^f(\rho_{\varepsilon,R}))$, $|\phi^R\big(x-m_\alpha^f(\rho_{\varepsilon,R})\big)|\le 1$, and utilize the conservation of mass and \eqref{ms} to get  
\begin{align}\label{xhrho}
\Big|\int_{\R^d}x\cdot h^R(x)\rho_{\varepsilon,R} dx\Big|
\le& \int_{\R^d}|x|^2\rho_{\varepsilon,R} dx
+ \int_{\R^d}|x|\big|m_\alpha^f(\rho_{\varepsilon,R})\big|\rho_{\varepsilon,R} dx\nonumber\\
\le& \int_{\R^d}|x|^2\rho_{\varepsilon,R} dx
+ \int_{\R^d} |x|^2 \rho_{\varepsilon,R} dx\Big(\int_{\R^d} \rho_{\varepsilon,R} dx\Big)^\frac{1}{2}\nonumber\\
\le& C\int_{\R^d}\big|x|^2\rho_{\varepsilon,R} dx
\end{align}
and
\begin{align}\label{h2rho}
\int_{\R^d}\big|h^R(x)\big|^2\rho_{\varepsilon,R} dx
\le&  2\left(\int_{\R^d}\big|x|^2\rho_{\varepsilon,R} dx + \int_{\R^d}\big|m_\alpha^f(\rho_{\varepsilon,R})\big|^2\rho_{\varepsilon,R} dx\right)\nonumber\\
\le&  2\int_{\R^d}\big|x|^2\rho_{\varepsilon,R} dx + 2\int_{\R^d}\big|x|^2\rho_{\varepsilon,R} dx\int_{\R^d}\rho_{\varepsilon,R} dx\nonumber \\
\le& C\int_{\R^d}\big|x|^2\rho_{\varepsilon,R} dx,
\end{align}
where $C$ is a constant independent of $R,\varepsilon$. Substituding \eqref{xhrho} and \eqref{h2rho} into \eqref{2m2}, we attain
\begin{align}\label{2m}
\frac{d}{dt}\int_{\R^d}|x|^2\rho_{\varepsilon,R} dx 
\le& C\int_{\R^d}|x|^2\rho_{\varepsilon,R} dx +\varepsilon^2.
\end{align}
Hence Grönwall's inequality indicates that there exists $M_2>0$ independent of $R, \varepsilon$ such that for any $t\in[0,T]$
\begin{align}\label{m2}
&\int_{\R^d}|x|^2\rho_{\varepsilon,R} dx \le  M_2.
\end{align} 
Furthermore, there exists $M_1>0$ independent of $R,\varepsilon$ such that
\begin{align}\label{m1}
\Big|\int_{\R^d}x\rho_{\varepsilon,R}&dx\Big| =\Big(\int_{\R^d}\rho_{\varepsilon,R} dx\Big)^\frac{1}{2}\Big(\int_{\R^d}|x|^2\rho_{\varepsilon,R} dx\Big)^\frac{1}{2}\le M_1.    
\end{align}

Then, we confirm the boundedness of the fourth moment. Multiplying equation \eqref{pdesR} by $|x|^4$ and integrating the resulting expression, we apply integration by parts to obtain the following estimate for arbitrary any $t\in[0,T]$
\begin{align}\label{4m1}
\frac{d}{dt}\int_{\R^d}|x|^4\rho_{\varepsilon,R} dx 
=& \lambda\int_{\R^d}|x|^4\nabla\cdot\big(h^R(x)\rho_{\varepsilon,R}\big)dx + \frac{\sigma^2}{2}\int_{\R^d}|x|^4\Delta\Big(\big(|h^R(x)|^2+\varepsilon^2\big)\rho_{\varepsilon,R}\Big)dx\nonumber\\
=& -4\lambda\int_{\R^d}|x|^2x\cdot h^R(x)\rho_{\varepsilon,R} dx + 2(2+d)\sigma^2\int_{\R^d}|x|^2\big(|h^R(x)|^2+\varepsilon^2\big)\rho_{\varepsilon,R} dx.
\end{align}
Again utilize $h^R=(x-m_\alpha^f(\rho_{\varepsilon,R}))\phi^R(x-m_\alpha^f(\rho_{\varepsilon,R}))$, $|\phi^R\big(x-m_\alpha^f(\rho_{\varepsilon,R})\big)|\le 1$, the conservation of mass and \eqref{ms}, we have 
\begin{align}\label{x3hR}
\Big|\int_{\R^d}|x|^2x\cdot h^R(x)\rho_{\varepsilon,R} dx\Big| 
\le& \int_{\R^d}|x|^4\rho_{\varepsilon,R} dx
+\int_{\R^d}|x|^3\big|m_\alpha^f(\rho_{\varepsilon,R})\big|\rho_{\varepsilon,R} dx\nonumber\\
\le& \int_{\R^d}|x|^4\rho_{\varepsilon,R} dx + \Big(\int_{\R^d}|x|^4\rho_{\varepsilon,R} dx\Big)^\frac{1}{2}\Big(\int_{\R^d} |x|^2\big|m_\alpha^f(\rho_{\varepsilon,R})\big|^2\rho_{\varepsilon,R} dx\Big)^\frac{1}{2}\nonumber\\
\le& 3\int_{\R^d}|x|^4\rho_{\varepsilon,R} dx + 2\Big(\int_{\R^d} |x|^2\rho_{\varepsilon,R} dx\Big)^2
\end{align}
and
\begin{align}\label{x2hR2}
\int_{\R^d}|x|^2|h^R(x)|^2\rho_{\varepsilon,R} dx
\le&2\Big(\int_{\R^d}|x|^4\rho_{\varepsilon,R} dx + \int_{\R^d}|x|^2|m_\alpha^f(\rho_{\varepsilon,R})|^2\rho_{\varepsilon,R} dx\Big)\nonumber\\
\le&2\int_{\R^d}|x|^4\rho_{\varepsilon,R} dx + 2\Big(\int_{\R^d}|x|^2\rho_{\varepsilon,R} dx\Big)^2.
\end{align}
Substituding \eqref{x3hR} and \eqref{x2hR2} into \eqref{4m1} and employing \eqref{m2}, we obtain
\begin{align}\label{4m2}
\frac{d}{dt}\int_{\R^d}|x|^4\rho_{\varepsilon,R} dx 
\le& C\int_{\R^d}|x|^4\rho_{\varepsilon,R} dx + \varepsilon^4.
\end{align}
By  Grönwall's inequality we have, for any $t\in[0,T]$, there exists $M_4>0$ independent of $R, \varepsilon$ such that 
\begin{align}\label{m4}
&\int_{\R^d}|x|^4\rho_{\varepsilon,R} dx \le M_4.
\end{align} 
Thus  \eqref{4m} holds.
\end{proof}

The following estimate on the spatial derivative of $\rho_{\varepsilon,R}$ is required in order to apply Aubin's lemma within the framework of the compactness argument. The proof is inspired by the argument presented in \cite[Lemma A.3]{trillos2024cb}.
\begin{lem}\label{lem3.3}
Suppose that $f$ meets Assumption \ref{ass}, and let the non-negative 
 initial data $\rho_0\in L^1\cap L^2(\R^d)\cap \mathcal{P}_4(\RR^d)$. 
$\rho_{\varepsilon,R}$ be the solution to the regularization problem \eqref{pdesR}. Then, for any fixed $T>0$, $r>0$, $\rho_{\varepsilon,R}$ satisfies the following spatial derivative estimate
\begin{align}\label{nrho} \|\nabla\rho_{\varepsilon,R}\|_{L^2(0,T;L^{\gamma}(B_r(0)))}\leq C,
\end{align}
and it holds
\begin{align}\label{rhoR2}
\|\rho_{\varepsilon,R}\|_{L^2(0,T;W^{1, \gamma}(B_r(0)))}\leq C,
\end{align}
where $1<\gamma<\frac{2d}{2+d}$  and $C$ is independent of $R,\varepsilon$.
\end{lem}
\begin{proof}
First, define $r$ satisfy $B_r(0)\subset B_R(m_\alpha^f(\rho_{\varepsilon,R}))$ for sufficiently large $R$. For any $t\in[0,T]$, utilize $|m_\alpha^f(\rho_{\varepsilon,R})| \le \Big|\int_{\R^d}x\rho_{\varepsilon,R}dx\Big|\le M_1$ and the Cauchy-Schwarz inequality to get
\begin{align}\label{rhor}
\int_{B_r(0)}|\nabla\rho_{\varepsilon,R}|^\gamma dx
&= \int_{B_r(0)}|\nabla\rho_{\varepsilon,R}|^\gamma \frac{|x- m_\alpha^f(\rho_{\varepsilon,R})|^\gamma}{|x- m_\alpha^f(\rho_{\varepsilon,R})|^\gamma} dx\nonumber\\
&\le \Big(\int_{B_r(0)} \frac{1}{|x- m_\alpha^f(\rho_{\varepsilon,R})|^{\gamma \widehat{q}}}dx\Big)^\frac{1}{\widehat{q}}
\Big(\int_{B_r(0)}|\nabla\rho_{\varepsilon,R}|^{\gamma \widehat{p}} |x- m_\alpha^f(\rho_{\varepsilon,R})|^{\gamma \widehat{p} }dx\Big)^\frac{1}{\widehat{p}},
\end{align}
where $1/\widehat{p} + 1/\widehat{q} = 1$. Taking $\widehat{p}=\frac{2}{\gamma}$ and $\widehat{q}=\frac{2}{2-\gamma}$, and using $1<\gamma<\frac{2d}{2+d}$, we have $\gamma \widehat{q}<d$. Hence it holds that for any $t\in[0,T]$
\begin{align}\label{1/x-m}
\int_{B_r(0)} \frac{1}{|x- m_\alpha^f(\rho_{\varepsilon,R})|^{\gamma \widehat{q}}}dx \le \int_{B_{r+M_1}(0)} \frac{1}{|x|^{\gamma \widehat{q}}}dx=S_{d}\int_0^{r+M_1}s^{d-1-\gamma \widehat{q}}ds= S_{d}\frac{(r+M_1)^{d-\gamma \widehat{q}}}{d-\gamma \widehat{q}},
\end{align}
where $S_{d}$ is the superficial area of the $d$-dimensional unit ball. Noticing that $B_r(0)\subset B_R(m_\alpha^f(\rho_{\varepsilon,R}))$, we have
\begin{align}\label{rhoxm}
\int_{B_r(0)}|\nabla\rho_{\varepsilon,R}|^2 |x- m_\alpha^f(\rho_{\varepsilon,R})|^2dx &\le\int_{B_R(m_\alpha^f(\rho_{\varepsilon,R}))}|\nabla\rho_{\varepsilon,R}|^2 |h^R(x)|^2dx\nonumber\\
&\le \int_{\R^d}|\nabla\rho_{\varepsilon,R}|^2 |h^R(x)|^2dx. 
\end{align}

Thus, substituding \eqref{1/x-m} and \eqref{rhoxm} into \eqref{rhor}, we obtain for any $t\in[0,T]$
\begin{align*}
\int_{B_r(0)}|\nabla\rho_{\varepsilon,R}|^\gamma dx \le C(r)\Big(\int_{\R^d}|\nabla\rho_{\varepsilon,R}|^2 |h^R(x)|^2dx\Big)^\frac{\gamma}{2},
\end{align*}
a.e. 
\begin{align}\label{nrhosr}
\|\nabla\rho_{\varepsilon,R}\|_{L^\gamma(B_r(0))}^2 \le C(r)\int_{\R^d}|\nabla\rho_{\varepsilon,R}|^2 |h^R(x)|^2dx.
\end{align}
Integrating on both sides of \eqref{nrhosr} over the time $[0,T]$ and using a result \eqref{L2T2} of Lemma \ref{lem3.1}, we obtain
\begin{align*}
&\int_0^T\|\nabla\rho_{\varepsilon,R}\|_{L^\gamma(B_r(0))}^2dt \le C\|\rho_0\|_{L^2(\R^d)}^2.
\end{align*}
This completes the proof of this Lemma.
\end{proof}

The following time derivative estimate for $\rho_{\varepsilon,R}$ will be essential.
\begin{lem}\label{lem3.4}
Assume that $f$ satisfies Assumption \ref{ass}, and let the non-negative 
 initial data $\rho_0\in L^1\cap L^2(\R^d)\cap \mathcal{P}_4(\RR^d)$ and $\rho_{\varepsilon,R}$ be the solution to the regularization problem \eqref{pdesR}. Then the time derivative of $\rho_{\varepsilon,R}$ satisfies the following uniform estimate
\begin{align}\label{trho}
\|\partial_t\rho_{\varepsilon,R}\|_{L^\infty(0,T;~W^{-2,\frac{4}{3}}_{loc}(\R^d))}\le C,
\end{align}
where $C$ is constant independent of $R, \varepsilon$.
\end{lem}
\begin{proof}
For any test function $\psi(x)\in C_c^\infty(\R^d)$ and $\Omega:=\mathrm{ssup}\psi$, we deduce that for any $t\in[0,T]$,
\begin{align}\label{ptrho}
&|<\partial_t\rho_{\varepsilon,R},\psi>| \nonumber\\
=& \Big|-\lambda\int_{\Omega}\nabla\psi\cdot h^R(x)\rho_{\varepsilon,R} dx + \frac{\sigma^2}{2}\int_{\Omega} \Delta\psi(|h^R(x)|^2 +\varepsilon^2\big)\rho_{\varepsilon,R}dx \Big|\nonumber\\
\le& \lambda\Big|\int_{\Omega}\nabla\psi\cdot h^R(x)\rho_{\varepsilon,R} dx\Big| 
+\frac{\sigma^2}{2}\Big|\int_{\Omega} \Delta\psi|h^R(x)|^2 \rho_{\varepsilon,R}dx \Big|
+\frac{\sigma^2\varepsilon^2}{2}\Big|\int_{\Omega} \Delta\psi\rho_{\varepsilon,R}dx \Big|\nonumber\\
=:& \mathcal{I}_1 + \mathcal{I}_2 +\mathcal{I}_3. \end{align}

Observe that for the first item $\mathcal{I}_1$, by the Hölder inequality we have
\begin{align*}
\Big|\int_{\Omega}\nabla\psi\cdot h^R(x)\rho_{\varepsilon,R}dx\Big|\le&
\| \phi^R\|_{L^\infty(\Omega)}\int_{\Omega}|\nabla\psi||x- m_\alpha^f(\rho_{\varepsilon,R})|\rho_{\varepsilon,R}dx \nonumber\\
\le& \int_{\Omega}|\nabla\psi||x |\rho_{\varepsilon,R}dx + \int_{\Omega}|\nabla\psi||m_\alpha^f(\rho_{\varepsilon,R})|\rho_{\varepsilon,R}dx\nonumber\\
\le& \|\nabla\psi\|_{L^\frac{8}{3}(\Omega)}\Big(\int_{\Omega}|x |^4\rho_{\varepsilon,R}dx\Big)^\frac{1}{4}\Big(\int_\Omega|\rho_{\varepsilon,R}|^2dx\Big)^\frac{3}{8} \nonumber\\
&+ \|\nabla\psi\|_{L^2(\Omega)}\|\rho_{\varepsilon,R}\|_{L^2(\Omega)}\int_{\Omega}|x|\rho_{\varepsilon,R}dx.
\end{align*}
Hence applying the boundedness of moments and \eqref{rhoR}, we get 
\begin{align}\label{Is1}
\mathcal{I}_1 = \lambda\Big|\int_{\Omega}\nabla\psi\cdot h^R(x)\rho_{\varepsilon,R}dx\Big|\leq C(\|\nabla\psi\|_{L^2(\Omega)} + \|\nabla\psi\|_{L^{\frac{8}{3}}(\Omega)}).
\end{align}

For $\mathcal{I}_2$, a simple computation gives that 
\begin{align*}
\Big|\int_{\Omega}\Delta\psi| h^R(x)|^2\rho_{\varepsilon,R}dx\Big|\le&
\| \phi^R\|_{L^\infty(\Omega)}^2\int_{\Omega}|\Delta\psi||x- m_\alpha^f(\rho_{\varepsilon,R})|^2\rho_{\varepsilon,R}dx \nonumber\\
\le& 2\int_{\Omega}|\Delta\psi||x |^2\rho_{\varepsilon,R}dx + 2\int_{\Omega}|\Delta\psi||m_\alpha^f(\rho_{\varepsilon,R})|^2\rho_{\varepsilon,R}dx\nonumber\\
\le& 2\|\Delta\psi\|_{L^4(\Omega)}\Big(\int_{\Omega}|x |^4\rho_{\varepsilon,R}dx\Big)^\frac{1}{2}\Big(\int_\Omega|\rho_{\varepsilon,R}|^2dx\Big)^\frac{1}{4} \nonumber\\
&+ 2\|\Delta\psi\|_{L^2(\Omega)}\|\rho_{\varepsilon,R}\|_{L^2(\Omega)}\int_{\Omega}|x|^2\rho_{\varepsilon,R}dx.
\end{align*}
Thus, by the  boundedness of moments and \eqref{rhoR}, $\mathcal{I}_2$ can be shown
\begin{align}\label{Is3}
\mathcal{I}_2 = \frac{\sigma^2}{2}\Big|\int_{\Omega}\Delta\psi| h^R(x)|^2\rho_{\varepsilon,R}dx\Big| \le C(\|\Delta\psi\|_{L^2(\Omega)} +\|\Delta\psi\|_{L^4(\Omega)}).
\end{align}

For the last term of \eqref{ptrho}, we have
\begin{align}\label{tlast}
\mathcal{I}_3 = \frac{\sigma^2\varepsilon^2}{2}\Big|\int_{\Omega} \Delta\psi\rho_{\varepsilon,R}dx \Big|
\le C(\varepsilon)\|\Delta\psi\|_{L^{2}(\Omega)}\|\rho_{\varepsilon,R}\|_{L^2(\Omega)}.
\end{align}

Substituting $\mathcal{I}_1$, $\mathcal{I}_2$ and $\mathcal{I}_3$ i.e. the inequalities \eqref{Is1}-\eqref{tlast} into \eqref{ptrho}, we obtain for any $t>0$
\begin{align*}
|<\partial_t\rho_{\varepsilon,R},\psi> |\le &C(\|\nabla\psi\|_{L^2(\Omega)}+  \|\nabla\psi\|_{L^\frac{8}{3}(\Omega)} +\|\Delta\psi\|_{L^2(\Omega)} + \|\Delta\psi\|_{L^4(\Omega)})
\le C\|\psi\|_{W^{2,4}(\Omega)}.
\end{align*}
Therefore, we have 
\begin{align}\label{no,t}
\|\partial_t\rho_{\varepsilon,R}\|_{W^{-2,\frac{4}{3}}_{loc}(\mathbb{R}^d)} \le C, \quad t>0,
\end{align}
where $C$ is independent of $R,\varepsilon$. Hence we derive the following result
\begin{align}
\|\partial_t\rho_{\varepsilon,R}\|_{L^\infty(0,T;~W_{loc}^{-2,\frac{4}{3}}(\mathbb{R}^d))}\le C.
\end{align}
This concludes the proof of Lemma \ref{lem3.4}.

\end{proof}

\section{The existence and uniqueness of solutions }
We are now in a position to prove our main theorem. 
It is important to note that the uniform estimates established in Lemmas~\ref{lem2.1}--\ref{lem3.4} are independent of the parameters $R$ and $\varepsilon$. 
Therefore, we may set $R := \frac{1}{\varepsilon}$ in the model~\eqref{pdesR}, so that the function $\rho_\varepsilon := \rho_{\varepsilon,1/\varepsilon}$ satisfies the following equation
\begin{align}\label{pdes1}
\begin{cases}
& \partial_t\rho_{\varepsilon,} = \lambda\nabla\cdot\big(h^{\varepsilon}(x)\rho_{\varepsilon}\big) + \frac{\sigma^2}{2}\Delta\big((|h^\varepsilon(x)|^2+\varepsilon^2)\rho_{\varepsilon}\big),\\
& \rho_{0\varepsilon}(x):=\rho_{0}(x)*J_{\varepsilon}(x),
\end{cases}
\end{align}
where $$h^\varepsilon(x)=(h^1, h^2,\cdots,h^d):=(x- m_\alpha^f(\rho_{\varepsilon}))\phi^\frac{1}{\varepsilon}(x- m_\alpha^f(\rho_{\varepsilon})).$$ 
For all $i,j,k\in[d]$, it is easy to see that
\begin{align}\label{hs}
\| h^i\|_{L^\infty(\R^d)}\le \frac{C}{\varepsilon},~~\|\partial_{x_j}h^i\|_{L^\infty(\R^d)}\le C \ \mathrm{and}\ \|\partial_{x_j,x_k}^2h^i\|_{L^\infty(\R^d)}\le C\varepsilon.
\end{align}
Moreover, the uniform estimates \eqref{rhoR}, \eqref{rhoR2}, \eqref{trho} and moment estimates \eqref{4m}, \eqref{m2} established in Lemmas \ref{lem2.1}-\ref{lem3.4} remain valid for equation \eqref{pdes1}. In other words, we obtain the following remark.
\begin{rem}\label{rem3.1}
Assume that $f$ satisfies Assumption \ref{ass}, and let the non-negative 
 initial data $\rho_0\in L^1\cap L^2(\R^d)\cap \mathcal{P}_4(\RR^d)$ and $\rho_{\varepsilon}$ be the solution to the regularization problem \eqref{pdes1}. Then for any fixed $T>0$, it holds that
\begin{align}
&\|\rho_\varepsilon\|_{L^\infty(0,T;L^2(\RR^d)}\le C,\label{l2estimate}\\
&\|\rho_\varepsilon\|_{L^2(0,T;L^2\cap W_{loc}^{1,\gamma}(\RR^d))}\le C,~~\gamma\in(1,\frac{2d}{2+d}),\label{nrhos}\\
&\|\partial_t\rho_{\varepsilon}\|_{L^2(0,T;W^{-2,\frac{4}{3}}_{loc}(\R^d))}\le C,\label{trhos}
\end{align}
and the estimates of moments
\begin{align}\label{m24}
\sup_{t\in[0,T]}\int_{\R^d}|x|^2\rho_\varepsilon dx \le M_2 ~~~\text{and}~\sup_{t\in[0,T]}\int_{\R^d}|x|^4\rho_\varepsilon dx \le M_4,
\end{align}
where $C, M_2$ and $M_4$ are constants independent of $\varepsilon$.
\end{rem}
Following Remark \ref{rem3.1}, for any fixed \( r > 0 \), there exists a subsequence of \( \{\rho_{\varepsilon}\} \) (not relabeled for convenience) and a limit function
\[
\rho^r \in L^2\big(0,T; L^2 \cap W^{1,\gamma}(B_r(0))\big),
\]
such that the following weak convergence holds
\[
\rho_{\varepsilon} \rightharpoonup \rho^r \quad \text{in} \ L^2\big(0,T; L^2 \cap W^{1,\gamma}(B_r(0))\big), \quad \text{for } \gamma \in \left(1, \frac{2d}{2 + d}\right).
\]

Furthermore, adopting the uniform estimates \eqref{nrhos} and \eqref{trhos}, and noticing that $W^{1,\gamma}(B_r(0))\hookrightarrow\hookrightarrow  L^q(B_r(0))$,  $\ q \in[1,\frac{d\gamma}{d-\gamma})\subset[1,2)$, and when $\max\{\frac{4d}{3d+12},1\}<\gamma<\frac{2d}{d+2}$, $\frac{4d}{3d+8}<q<\frac{d\gamma}{d-\gamma}$, we have
 $ L^{q}(B_r(0)) \hookrightarrow W^{-2,\frac{4}{3}}(B_r(0))$. Thus, the Lions-Aubin lemma implies that $\rho_\varepsilon$ is relatively compact in $L^2(0,T; L^q(B_r(0)))$ for all $q\in [\frac{4d}{3d+8},\frac{d\gamma}{d-\gamma})$. Since the above spaces is of finite measure, we obtain the relative
compactness of $\rho_\varepsilon$ in $L^2(0,T; L^q(B_r(0)))$ for all $q\in[1, \frac{d\gamma}{d-\gamma})$. Thus we can
extract a subsequence of $\{\rho_\varepsilon\}$ (still denoted by $\rho_\varepsilon$ without relabeling) such that, as $\varepsilon\to 0$, the following convergence holds
\begin{align}\label{Rtovar}
\rho_{\varepsilon} \to \rho^r\qquad \mbox{in} \ L^2(0,T; L^q(B_r(0))), ~~q\in[1, \frac{d\gamma}{d-\gamma}).
\end{align}

Next, we verify that $\rho_\varepsilon$ converges strongly to some $\rho$ in $L^{q^*}(0,T; L^q(\mathbb{R}^d))$ with $q\in [1,\frac{d\gamma}{d-\gamma})$ $~\text{and}~q^*=\frac{q}{2-q}$.

\begin{lem}\label{lem3.1}
Let $\{\rho_\varepsilon\}$ be a sequence of solutions to PDE \eqref{pdes1} satisfying uniform estimates listed in Remark \ref{rem3.1}. Then $\{\rho_\varepsilon\}$ is relative compact in $L^2(0,T;L^q(\R^d)), ~q \in[1,\frac{d\gamma}{d-\gamma})\subset[1,2)$. Moreover, there exists a subsequence $\rho_\varepsilon$ (not relabeled) and a function $\rho \in L^\infty(0,T; L^2(\mathbb{R}^d))$ such that
\begin{align}
&\rho_{\varepsilon} \to \rho\quad \textup{in} \ L^{q*}(0,T; L^q(\mathbb{R}^d)), ~~q\in [1,\frac{d\gamma}{d-\gamma}) ~\textup{and}~q^*=\frac{q}{2-q},\label{rhouniform}\\
&\rho_{\varepsilon} \rightharpoonup \rho\quad\textup{in} \ L^2\big(0,T;L^2\cap W_{loc}^{1,\gamma}(\RR^d)\big),~~ \gamma\in(1,\frac{2d}{2+d}),\label{Rntovar2}\\
&\rho_{\varepsilon} \stackrel{*}{\rightharpoonup}\rho\quad\textup{in} \ L^\infty\big(0,T;L^2(\RR^d)\big),\label{Rntovar3}\\
&\partial_t\rho_{\varepsilon} \rightharpoonup \partial_t\rho\quad\textup{in} \ L^2\big(0,T;W^{-2,\frac{4}{3}}_{loc}(\R^d)\big), \label{Rntovar4}\\
&\sup_{t\in[0,T]}\int_{\R^d}|x|^4\rho dx\leq M_4\,\label{M4}.
\end{align}
\end{lem}
\begin{proof}
It follows from \eqref{Rtovar} that we can obtain that a subsequence $\{\rho_{\varepsilon_n}^1\}$
converges almost everywhere to $\rho^r$ as $n\to \infty$, namely
\begin{align*}
\rho_{\varepsilon_n}^1(t,x) \to \rho^r(t,x)\qquad \, ~~ \text {a.e.}  ~(t,x)\in[0,T]\times B_r(0).
\end{align*}
Next from subsequence $\{\rho_{\varepsilon_n}^1\}$ we extract a further subsequence $\{\rho_{\varepsilon_n}^2\}$ such that
\begin{align*}
\rho_{\varepsilon_n}^2(t,x) \to \rho^{2r}(t,x)\qquad \, ~~ \text {a.e.}  ~(t,x)\in[0,T]\times B_{2r}(0).
\end{align*}
Continuing with this process we have
\begin{align*}
& \rho_{\varepsilon_1}^1, \rho_{\varepsilon_2}^1,\rho_{\varepsilon_3}^1,\cdots, \rho_{\varepsilon_n}^1\to \rho^{r}\qquad \, ~~ \text {a.e.}  ~(t,x)\in[0,T]\times B_{r}(0),\\
& \rho_{\varepsilon_1}^2, \rho_{\varepsilon_2}^2,\rho_{\varepsilon_3}^2,\cdots, \rho_{\varepsilon_n}^2\to \rho^{2r}\qquad \, ~~ \text {a.e.}  ~(t,x)\in[0,T]\times B_{2r}(0),\\
& \rho_{\varepsilon_1}^3, \rho_{\varepsilon_2}^3,\rho_{\varepsilon_3}^3,\cdots, \rho_{\varepsilon_n}^3\to \rho^{3r}\qquad \, ~~ \text {a.e.}  ~(t,x)\in[0,T]\times B_{3r}(0),\\
&\cdots \cdots\quad  \cdots \cdots\quad  \cdots \cdots\\
& \rho_{\varepsilon_1}^n, \rho_{\varepsilon_2}^n,\rho_{\varepsilon_3}^n,\cdots, \rho_{\varepsilon_n}^n\to \rho^{nr}\qquad \, ~~ \text {a.e.}  ~(t,x)\in[0,T]\times B_{nr}(0).
\end{align*}
where $\{\rho_{\varepsilon_n}^{m+1}\}$ is a subsequence of $\{\rho_{\varepsilon_n}^{m}\}$.
For any $x\in \R^d$ and $t\in[0,T]$, we choose some $k\in N$
such that $x\in B_{kr}(0)$ and define $\rho(t,x):=\rho^{kr}(t,x)$. This definition is well-defined, since for any $x\in B_{kr}(0)\subset B_{\ell r}(0),~l>k$, we have $\rho^{kr}(\cdot,x)=\rho^{\ell r}(\cdot,x)$
on $B_{kr}(0)$ by the construction above. Consequently, using Cantor's diagonal argument we have a subsequence $\{\rho_{\varepsilon_n}^n\}_{n\geq 1}$ converging
almost everywhere to $\rho$ on $[0,T]\times \R^d$.

Now let us use the notation $\{\rho_{\varepsilon}\}$ for the subsequence $\{\rho_{\varepsilon_n}^n\}_{n\geq 1}$ again.
Then, by Fatou's lemma and the moment bound \eqref{m24}, we obtain
\begin{align}\label{t4m}
\sup_{t\in[0,T]}\int_{\R^d}|x|^4\rho dx=\sup_{t\in[0,T]}\int_{\R^d}\liminf_{\varepsilon\to 0}|x|^4\rho_\varepsilon dx\le \sup_{t\in[0,T]}\liminf_{\varepsilon\to 0}\int_{\R^d}|x|^4\rho_\varepsilon dx\le M_4\,.
\end{align}
Moreover, it follows from Fatou's lemma and Lemma \ref{lem3.1}  that $\rho \in L^\infty (0,T;L^2(\RR^d))$.

It remains to prove that $\rho_\varepsilon\to \rho$ in $L^{q^*}([0, T]; L^q(\R^d))$ as $\varepsilon\to0$. To do this, let us compute
\begin{align*}
\int_0^T\int_{\R^d}|\rho_\varepsilon - \rho| dxdt
\le&
\int_0^T\int_{B_r(0)}|\rho_\varepsilon - \rho| dxdt+
\int_0^T\int_{\R^d/B_r(0)}|\rho_\varepsilon - \rho| dxdt\\
\le&
\int_0^T\int_{B_r(0)}|\rho_\varepsilon - \rho| dxdt+
\frac{1}{r^4}\int_0^T\int_{\R^d/B_r(0)}|x|^4|\rho_\varepsilon - \rho| dxdt.
\end{align*}
Taking $\varepsilon\to 0$ and then $r\to\infty$, we find a subsequence $\rho_\varepsilon$ (not relabeled), which converges in $L^1(0,T;L^1(\R^d))$. Moreover, the estimate \eqref{l2estimate} implies that $\rho_\varepsilon,\rho\in L^\infty(0,T;L^2(\R^d))$. By the interpolation inequality, we show
\begin{align}\label{theta}
\|\rho_\varepsilon-\rho\|_{L^q(\R^d)}\le\|\rho_\varepsilon-\rho\|_{L^1(\R^d)}^{1-\theta}\|\rho_\varepsilon-\rho\|_{L^{2}(\R^d)}^\theta,
\end{align}
where $\frac{1}{q}=1-\theta+\frac{\theta}{2},\theta\in[0,1)$, that is $\theta=\frac{2q-2}{q}$ and $1\leq q<\frac{d\gamma}{d-\gamma}<2$.
Multiplying both sides of the inequality \eqref{theta} by $q^*=\frac{1}{1-\theta}$ and integrating over the time $[0,T]$, we obtain
\begin{align*}
\int_0^T\|\rho_\varepsilon - \rho\|_{L^{q}(\R^d)}^{q^*} dt\le\|\rho_\varepsilon - \rho\|_{L^\infty(0,T;L^2(\R^d))}^{\theta q^*}\int_0^T\|\rho_\varepsilon - \rho\|_{L^1(\R^d)}dx,
\end{align*} 
Thus, $\rho_\varepsilon$ converges strongly to $\rho$ in $L^{q^*}(0,T; L^q(\mathbb{R}^d))$ and $q^*=\frac{q}{2-q}$. 

Moreover, by the estimates in Remark \ref{rem3.1}, we have
the weak convergence relations \eqref{Rntovar2}- \eqref{Rntovar4}.
\end{proof}

Using Lemma \ref{lem3.1} , we can further obtain that a subsequence $\rho_\varepsilon$ (not relabeled)
converges to $\rho$ in $L^q(\RR^d)$ a.e. $t\in[0,T]$, namely
\begin{align}\label{Rtovar1}
\rho_{\varepsilon} \to \rho\qquad \mbox{in} \  L^q(\RR^d), ~~ \text {a.e.}  ~t\in[0,T]\quad \mbox{as }\varepsilon\to 0.
\end{align}
Subsequently, we establish a pivotal lemma that characterizes the convergence between $\rho_{\varepsilon}$ and $\rho$ in the 2-Wasserstein distance. This result is essential for demonstrating the convergence of $m_\alpha^f(\rho_{\varepsilon})$ to $m_\alpha^f(\rho)$ as $\varepsilon \to 0$, and its proof follows an argument analogous to that of~\cite[Lemma A.7]{gerber2023mean}.
\begin{lem}\label{lem3.6}
Suppose that $\rho_{\varepsilon}, \rho\in\mathcal{P}_4(\R^d)$ are density functions that satisfy $\rho_{\varepsilon}$ converging to $\rho$ in $L^q(\R^d)$ a.e. $t\in[0,T],\ q\in [1,\frac{d\gamma}{d-\gamma})\subset[1,2)$. Then $\rho_{\varepsilon}$ converges to $\rho$ in 2-Wasserstein distance, i.e.,
\begin{align}\label{W2}
\lim_{\varepsilon\to 0} W_2(\rho_{\varepsilon}, \rho)=0\qquad \textup{a.e.}~~ t\in[0,T].    
\end{align}
\end{lem}
\begin{proof}
 The conditions of convergence in the 2-Wasserstein distance can be seen in \cite[Chapter 6]{villani2008grundlehren}. We mainly show that $\int_{\R^d}|x|^2\rho_{\varepsilon} dx \to \int_{\R^d}|x|^2\rho dx$ as
$\varepsilon\to 0$ and also $\rho_{\varepsilon}$ converges $\rho$ weakly \text{a.e.}~~ $t\in[0,T]$. 

First, for any fixed $r > 0$, by the boundedness of the fourth moment $\rho_{\varepsilon}$ in \eqref{m24}, we have
\begin{align}\label{2m0}
\int_{\R^d/B_{r}(0)}|x|^2\rho_{\varepsilon}(x)dx
\le &\frac{1}{r^2}\int_{\R^d/B_r(0)}|x|^4\rho_{\varepsilon}(x)dx \le \frac{M_4}{r^2}.
\end{align}
Then if for every $\delta>0$, there exists a $r_\delta :=\sqrt{\frac{M_4}{\delta}}>0$
such that whenever $r\ge r_\delta$, it follows that
\begin{align}\label{r2m0}
\int_{\R^d/B_{r_\delta}(0)}|x|^2\rho_{\varepsilon}(x)dx\le  \delta.  
\end{align}
Similarly, since we know that for
$\int_{\R^d}|x|^4\rho(x)dx\le M_4$ in \eqref{t4m}, the same computation as in \eqref{2m0} holds for $\rho$, and we
also have $\int_{\R^d/B_{r_\delta}(0)}|x|^2\rho(x)dx\le  \delta$.

Then, we obtain a strong convergent of the second moment. We compute
\begin{align}\label{ssm}
&\Big|\int_{\R^d}|x|^2\rho_{\varepsilon}(x)dx - \int_{\R^d}|x|^2\rho(x)dx\Big|\nonumber\\
\le&\Big|\int_{B_{r_\delta}(0)}|x|^2(\rho_{\varepsilon}(x) - \rho(x))dx\Big|+\Big|\int_{\R^d/B_{r_\delta}(0)}|x|^2\rho_{\varepsilon}(x)dx - \int_{\R^d/B_{r_\delta}(0)}|x|^2\rho(x)dx\Big|\nonumber\\
\le& \Big(\int_{B_{r_\delta}(0)}|x|^\frac{2q}{q-1}dx\Big)^\frac{q-1}{q}\|\rho_{\varepsilon}-\rho\|_{L^q(\R^d)}
+ \Big|\int_{\R^d/B_{r_\delta}(0)}|x|^2\rho_{\varepsilon}(x)dx\Big|+ \Big|\int_{\R^d/B_{r_\delta}(0)}|x|^2\rho(x)dx\Big|\nonumber\\
\le& C_{r_\delta}\|\rho_{\varepsilon}-\rho\|_{L^q(\R^d)} + C(T)\delta.
\end{align}
By the fact that $\lim_{\varepsilon\to 0}\|\rho_{\varepsilon} - \rho\|_{L^q(\R^d)} = 0$ a.e. $t\in[0,T]$ in \eqref{rhouniform}, we have 
\begin{align*}
\mathop{\lim\sup}_{\varepsilon\to 0}\Big|\int_{\R^d}|x|^2\rho_{\varepsilon}(x)dx - \int_{\R^d}|x|^2\rho(x)dx\Big| \leq C(T)\delta\qquad\text{a.e.}~~ t\in[0,T].
\end{align*}
By the arbitrariness of $\delta$, we know that $\rho_{\varepsilon}$ converges to $\rho$ at the second moment a.e. $t\in[0,T]$. 

Next, we show that $\rho_{\varepsilon}$  converges weakly to $\rho$ a.e. $t\in[0,T]$ as probability density. For any $\delta> 0$ and $\Psi\in\mathcal{C}_b(\R^d)$, with the similar computation as \eqref{2m0}, we have
\begin{align*}
    \int_{\R^d/B_{r_\delta}(0)}|\Psi(x)|\rho_{\varepsilon}(x)dx\le C\delta, \qquad \int_{\R^d/B_{r_\delta}(0)}|\Psi(x)|\rho(x)dx\le C\delta.
\end{align*}
Thus, we obtain
\begin{align*}
&\Big| \int_{\R^d}\Psi(x)\rho_{\varepsilon}(x)dx -  \int_{\R^d}\Psi(x)\rho(x)dx\Big|\\
\le&\Big|\int_{B_{r_\delta}(0)}\Psi(x)(\rho_{\varepsilon}(x) - \rho(x))dx\Big|+ \Big|\int_{\R^d/B_{r_\delta}(0)}\Psi(x)\rho_{\varepsilon}(x)dx - \int_{\R^d/B_{r_\delta}(0)}\Psi(x)\rho(x)dx\Big|\\
\le& \Big(\int_{B_{r_\delta}(0)}\Psi(x)^\frac{q}{q-1}dx\Big)^\frac{q-1}{q}\|\rho_{\varepsilon}-\rho\|_{L^q(\R^d)} + \Big|\int_{\R^d/B_{r_\delta}(0)}\Psi(x)\rho_{\varepsilon}(x)dx\Big|+ \Big|\int_{\R^d/B_{r_\delta}(0)}\Psi(x)\rho(x)dx\Big|\\
 \le& C_{r_\delta}\|\rho_{\varepsilon}-\rho\|_{L^q(\R^d)} + C\delta\qquad \text{a.e}~ t\in[0,T], 
\end{align*}
which implies that
\begin{align*}
\mathop{\lim\sup}_{\varepsilon\to 0 }\Big| \int_{\R^d}\Psi(x)\rho_{\varepsilon}(x)dx -  \int_{\R^d}\Psi(x)\rho(x)dx\Big|\leq C\delta\qquad\text{a.e.}~~ t\in[0,T].
\end{align*}
Again, by the arbitrariness of $\delta>0$ and $\Psi\in\mathcal{C}_b(\R^d)$, we find that $\rho_{\varepsilon}$ weakly converges to $\rho$ \text{a.e.}~~ $t\in[0,T]$.

In conclusion, $\rho_\varepsilon$ converges to $\rho$ under the 2-Wasserstein distance for a.e. $t\in[0,T]$.
The lemma holds.

\end{proof}

{\bf The proof of Theorem \ref{thm1}}.
Taking any test function $\varphi\in\mathcal{C}_c^\infty((0,T)\times\R^d)$ and $\Omega:=\mathrm{supp}\varphi$.  Since $\rho_{\varepsilon}$ satisfies the equation \eqref{pdes1}, the following integral equality holds
\begin{align*}
-\int_0^T\int_{\R^d}\left(\partial_t\rho_{\varepsilon}\varphi
+\frac{\varepsilon^2\sigma^2}{2}\nabla\rho_{\varepsilon}\cdot\nabla\varphi \right)dxdt =&
\int_0^T\int_{\R^d}\lambda\big(h^\varepsilon(x)\rho_{\varepsilon}\big)\cdot\nabla\varphi+
\frac{\sigma^2}{2}|h^\varepsilon|^2 \nabla\rho_{\varepsilon}\cdot\nabla\varphi dxdt \\
&+\frac{\sigma^2}{2}\int_0^T\int_{\R^d}\nabla |h^\varepsilon|^2\rho_{\varepsilon}\cdot\nabla\varphi dxdt.
\end{align*}

To prove that $\rho$ is a weak solution,  based on the weak convergence relations \eqref{Rntovar2} and \eqref{Rntovar4} of $\rho_\varepsilon$ given by Lemma \ref{lem3.1}, we obtain the following convergence for the left-hand side of the equality as $\varepsilon \to 0$,
\begin{align}\label{i0}
\int_0^T\int_{\R^d}\left(\partial_t\rho_{\varepsilon}\varphi
+\frac{\varepsilon^2\sigma^2}{2}\nabla\rho_{\varepsilon}\cdot\nabla\varphi\right)dxdt \to \int_0^T\int_{\R^d}\partial_t\rho\varphi dxdt.
\end{align}
We now examine the right-hand side of the equation through the following convergence results:
\begin{itemize}
    \item [(i)]$\int_0^T\int_{\RR^d}\big(h^\varepsilon(x)\rho_{\varepsilon}\big)\cdot\nabla\varphi dxdt\to \int_0^T\int_{\RR^d}\big((x - m_\alpha^f(\rho))\rho\big)\cdot\nabla\varphi dxdt $,\\
    \item  [(ii)]$\int_0^T\int_{\RR^d}|h^\varepsilon|^2 \nabla\rho_{\varepsilon}\cdot\nabla\varphi dxdt \to \int_0^T\int_{\RR^d}|x - m_\alpha^f(\rho)|^2\nabla\rho\cdot\nabla\varphi dxdt $,\\
    \item [(iii)]$\int_0^T\int_{\RR^d}\nabla |h^\varepsilon|^2\rho_{\varepsilon}\cdot\nabla\varphi dxdt \to \int_0^T\int_{\RR^d}\nabla \big|(x - m_\alpha^f(\rho))\big|^2\rho\cdot\nabla\varphi dxdt $.
\end{itemize} 
Since the test function $\varphi\in\mathcal{C}_c^\infty([0,T]\times\R^d)$ has compact support in $\R^d$, we may restrict our spatial integration to a bounded domain $\Omega$. Consequently, it suffices to examine the convergence properties on $[0,T]\times\Omega$. Prior to establishing the three convergence relations, we first present the following key lemma.
\begin{lem}\label{lem3.7}
Suppose that $f$ satisfies Assumption \ref{ass} and $\rho_{\varepsilon}, \rho\in\mathcal{P}_4(\R^d)$ are density functions that satisfy $\rho_{\varepsilon}$ converging to $\rho$ in $L^q(\R^d)$ a.e. $t\in[0,T],\ q\in [1,\frac{d\gamma}{d-\gamma})\subset[1,2)$. Then for all $\varepsilon>0$ and for all $p\geq 1$, the following convergence relations hold
\begin{align}
&h^\varepsilon(x)\to(x - m_\alpha^f(\rho))\quad \mathrm{for~any} \ x\in\Omega\quad \text{a.e.} ~t\in[0,T],  \label{h-m} \\
&|h^\varepsilon(x)|^2\to|x - m_\alpha^f(\rho)|^2\quad \mathrm{for~any} \ x\in\Omega\quad \text{a.e.} ~t\in[0,T],\label{h2-m2}\\
&\nabla h^\varepsilon(x)\to\nabla (x - m_\alpha^f(\rho))\quad \mathrm{for~any} \ x\in\Omega\quad \text{a.e.} ~t\in[0,T],  \label{gh-gm} \\
&\nabla(|h^\varepsilon(x)|^2)\to \nabla(|x - m_\alpha^f(\rho)|^2)\quad \mathrm{for~any} \ x\in\Omega\quad \text{a.e.} ~t\in[0,T].\label{gh2-gm2} 
\end{align}   
\end{lem}
\begin{proof}
First, we prove \eqref{h-m} holds true. Indeed,
\begin{align}\label{hRxm}
|h^\varepsilon(x) - (x - m_\alpha^f(\rho))|=& \Big|(x - m_\alpha^f(\rho_{\varepsilon}))\phi^\frac{1}{\varepsilon}(x - m_\alpha^f(\rho_{\varepsilon}))-(x - m_\alpha^f(\rho))\Big|\nonumber\\
\le& \Big|\big(m_\alpha^f(\rho) -m_\alpha^f(\rho_\varepsilon)\big)\phi^\frac{1}{\varepsilon}(x - m_\alpha^f(\rho_\varepsilon))\Big|\nonumber\\
& + \Big|(x - m_\alpha^f(\rho))\Big(\phi^\frac{1}{\varepsilon}(x - m_\alpha^f(\rho_{\varepsilon})) - \phi^\frac{1}{\varepsilon}(x - m_\alpha^f(\rho))\Big) \Big|\nonumber\\
& + \Big|(x - m_\alpha^f(\rho))\phi^\frac{1}{\varepsilon}(x - m_\alpha^f(\rho)) - (x - m_\alpha^f(\rho)) \Big|.
\end{align} 
The estimates for the first and third terms are immediate. For the second term, we compute
\begin{align}\label{xmphiR}
&\Big|(x - m_\alpha^f(\rho))\Big(\phi^\frac{1}{\varepsilon}(x - m_\alpha^f(\rho_{\varepsilon})) - \phi^\frac{1}{\varepsilon}(x - m_\alpha^f(\rho))\Big) \Big|\nonumber\\
\le& \big|x - m_\alpha^f(\rho)\big|
\big|\phi^\frac{1}{\varepsilon}(x - m_\alpha^f(\rho_{\varepsilon})) - \phi^\frac{1}{\varepsilon}(x - m_\alpha^f(\rho)) \big|\nonumber\\
\le& \big|x - m_\alpha^f(\rho)\big|
|\nabla\phi(\xi)| \Big|{\varepsilon}|x - m_\alpha^f(\rho_{\varepsilon})| - {\varepsilon}|x - m_\alpha^f(\rho)|\Big|\nonumber\\
\le& \varepsilon C(\Omega)\big| m_\alpha^f(\rho_{\varepsilon}) -  m_\alpha^f(\rho) \big|.
\end{align}
Note that the second inequality applies the Mean Value Theorem, where $\xi$ lies between $\varepsilon|x - m_\alpha^f(\rho_{\varepsilon})|$ and $\varepsilon|x - m_\alpha^f(\rho)|$. 
Putting \eqref{xmphiR} into \eqref{hRxm}, and combining \eqref{W2}, we obtain for any $x\in\Omega$, as $\varepsilon\to 0$
\begin{align}\label{h1-m1}
\Big|h^\varepsilon(x)-(x - m_\alpha^f(\rho))\Big| 
\le& |\phi^\frac{1}{\varepsilon}|\big|m_\alpha^f(\rho_\varepsilon) - m_\alpha^f(\rho) \big| + \varepsilon C(\Omega)\Big| m_\alpha^f(\rho_{\varepsilon,}) -  m_\alpha^f(\rho) \Big|\nonumber\\
&+ \Big|(x - m_\alpha^f(\rho))(\phi^\frac{1}{\varepsilon} - 1)\Big|\nonumber\\
\le & W_2(\rho_{\varepsilon},\rho) + C(\Omega)\big|\phi^\frac{1}{\varepsilon} - 1\big|\to 0\qquad \text{a.e.}~ t\in[0,T].
\end{align} 

Effortlessly, by the above \eqref{h1-m1} we check as $\varepsilon\to 0$
\begin{align*}
&\Big||h^\varepsilon(x)|^2-|x - m_\alpha^f(\rho)|^2\Big|\\
=&\Big|h^\varepsilon(x)+(x - m_\alpha^f(\rho))\Big|\Big|h^\varepsilon(x)-(x - m_\alpha^f(\rho))\Big|\\
\le& C(\Omega)\Big|h^\varepsilon(x)-(x - m_\alpha^f(\rho))\Big|\to 0
\quad \ x\in\Omega\quad \text{a.e.} ~t\in[0,T].
\end{align*}
Hence, the limit relation \eqref{h2-m2} holds.

Next, we prove the convergence of the spatial derivative \eqref{gh-gm}. For simplicity, we need only prove the component form holds. For any $i=1,\cdots,d$, we have
\begin{align*}
&|\partial_{x_i} h^\varepsilon(x) - \partial_{x_i}(x - m_\alpha^f(\rho))|\nonumber\\
=& \Big|\partial_{x_i}\Big((x - m_\alpha^f(\rho_{\varepsilon}))\phi^\frac{1}{\varepsilon}(x - m_\alpha^f(\rho_{\varepsilon}))\Big)-\partial_{x_i}(x - m_\alpha^f(\rho))\Big|\nonumber\\
=&|\phi^\frac{1}{\varepsilon}(x - m_\alpha^f(\rho_{\varepsilon})) + (x_i - (m_\alpha^f(\rho_{\varepsilon}))_i)\partial_{x_i}\phi^\frac{1}{\varepsilon}(x - m_\alpha^f(\rho_{\varepsilon}))-1|\nonumber\\
\le&|\phi^\frac{1}{\varepsilon}(x - m_\alpha^f(\rho_{\varepsilon}))- \phi^\frac{1}{\varepsilon}(x - m_\alpha^f(\rho))| +|\phi^\frac{1}{\varepsilon}(x - m_\alpha^f(\rho))-1| +|x - m_\alpha^f(\rho_{\varepsilon})||\partial_{x_i}\phi^\frac{1}{\varepsilon}|\nonumber\\
\le& C\varepsilon\big| m_\alpha^f(\rho_{\varepsilon}) -  m_\alpha^f(\rho) \big|+|\phi^\frac{1}{\varepsilon}(x - m_\alpha^f(\rho))-1| +C(\Omega)\varepsilon
. 
\end{align*} 
This leads to the following result
\begin{align}\label{nc1}
\Big|\nabla h^\varepsilon(x) -
\nabla(x - m_\alpha^f(\rho))\Big|\to 0 \quad \ x\in\Omega\quad \text{a.e.} ~t\in[0,T]~~ \text{as} ~\varepsilon\to 0.
\end{align}

It then follows easily that the squared derivative converges \eqref{gh2-gm2}
\begin{align*}
&\Big|\nabla(|h^\varepsilon(x)|^2)-\nabla(|x - m_\alpha^f(\rho)|^2)\Big|\\
=&2\Big|h^\varepsilon(x)(\nabla(h^\varepsilon(x)))^\intercal-(x - m_\alpha^f(\rho))(\nabla(x - m_\alpha^f(\rho)))^\intercal\Big|\\
\le &2|h^\varepsilon(x)-(x - m_\alpha^f(\rho))||\nabla(h^\varepsilon(x))| + 2|x - m_\alpha^f(\rho)||\nabla(h^\varepsilon(x))-\nabla(x - m_\alpha^f(\rho))\Big|\\
\le &C|h^\varepsilon(x)-(x - m_\alpha^f(\rho))| + C(\Omega)|\nabla(h^\varepsilon(x))-\nabla(x - m_\alpha^f(\rho))\Big|.
\end{align*}
Given the two convergence relations \eqref{h1-m1} and \eqref{nc1}, we conclude that as $\varepsilon\to 0$
\begin{align}
\Big|\nabla (|h^\varepsilon(x)|^2)-
\nabla(|x - m_\alpha^f(\rho)|^2)\Big|\to 0\quad \ x\in\Omega\quad \text{a.e.} ~t\in[0,T].
\end{align}
The proof of this theorem is now completed.
\end{proof}

Now, we turn our attention to proving the three convergence relationships separately.  

\begin{proof*}
(i) We check
\begin{align*} &\Big|\int_0^T\int_{\Omega}\big(h^\varepsilon(x)\rho_{\varepsilon} - (x - m_\alpha^f(\rho))\rho\big)\cdot\nabla\varphi dxdt\Big|\\
\le &\Big|\int_0^T\int_{\Omega}h^\varepsilon(x)(\rho_{\varepsilon} - \rho)\cdot\nabla\varphi dxdt\Big| + \Big|\int_0^T\int_{\Omega}\rho \Big(h^\varepsilon(x) - \big(x- m_\alpha^f(\rho)\big)\Big) \cdot\nabla\varphi dxdt\Big|\\
\le &\big\| h^\varepsilon(x)\cdot\nabla\varphi\big\|_{L^{(q^*)'}(0,T;L^{q'}(\Omega))}\big\|\rho_{\varepsilon} - \rho\big\|_{L^{q^*}(0,T;L^q(\Omega))}\\
&+ \|h^\varepsilon(x) - \big(x- m_\alpha^f(\rho)\big)\|_{L^2(0,T;L^2(\Omega))}\|\nabla\varphi\cdot\rho\|_{L^2(0,T;L^2(\Omega))} 
\to 0 \qquad \mbox{as} \ \varepsilon\to 0,
\end{align*}
where $1/{q'} + 1/q = 1$ and $1/{(q^*)'} + 1/{q^*} = 1$. 
Here, the first term converges as $
\rho_{\varepsilon} \to \rho$ in $ L^{q*}(0,T; L^q(\mathbb{R}^d)), ~~q\in [1,\frac{d\gamma}{d-\gamma})$ and $q^*=\frac{q}{2-q}$ in \eqref{rhouniform}. And for the second term, since for any $(t,x)\in[0,T]\times\Omega$, the sequence $h^\varepsilon$ convergence almost everywhere to $(x-m_\alpha^f(\rho))$ in \eqref{h-m} and $\big|h^\varepsilon(x) - \big(x- m_\alpha^f(\rho)\big)\big|\le C(\Omega)$, apply the Dominated Convergence Theorem, we have 
\begin{align*}
\lim_{\varepsilon\to 0}\int_0^T\int_{\Omega}\big|h^\varepsilon(x) - \big(x- m_\alpha^f(\rho)\big)\big|^2dxdt =0.    
\end{align*}
Hence there holds
\begin{align}\label{i}
&\Big|\int_0^T\int_{\Omega}\big(h^\varepsilon(x)\rho_{\varepsilon} - (x - m_\alpha^f(\rho))\rho\big)\cdot\nabla\varphi dxdt\Big|
\to 0 \qquad \mbox{as} \ \varepsilon\to 0.    
\end{align}

(ii) We verify
\begin{align}\label{JJ}
&\Big|\int_0^T\int_{\Omega}\big(|h^\varepsilon(x)|^2\nabla\rho_{\varepsilon} - |x -m_\alpha^f(\rho)|^2\nabla\rho\big)\cdot \nabla\varphi dxdt\Big|\nonumber\\
\le & \Big|\int_0^T\int_{\Omega}\big(|h^\varepsilon(x)|^2 - |x- m_\alpha^f(\rho)|^2\big) \nabla\rho_{\varepsilon}\cdot\nabla\varphi dxdt\Big|\nonumber\\
&+ \Big|\int_0^T\int_{\Omega}|x- m_\alpha^f(\rho)|^2\nabla(\rho_{\varepsilon} - \rho)\cdot\nabla\varphi  dxdt\Big|\nonumber\\
\le & \||h^\varepsilon(x)|^2 - |x- m_\alpha^f(\rho)|^2\|_{L^2(0,T;L^{\frac{\gamma}{\gamma-1}}(\Omega))}\|\nabla\varphi\cdot\nabla\rho_\varepsilon\|_{L^2(0,T;L^{\gamma}(\Omega))}\nonumber\\
&+ \Big|\int_0^T\int_{\Omega}|x- m_\alpha^f(\rho)|^2\nabla(\rho_{\varepsilon} - \rho)\cdot\nabla\varphi  dxdt\Big|\nonumber\\
=:&J_1+J_2.
\end{align}
For $J_1$, since for any $(t,x)\in[0,T]\times\Omega$, the sequence $|h^\varepsilon|^2$ convergence almost everywhere to $|x-m_\alpha^f(\rho)|^2$ in \eqref{h2-m2} and $\big||h^\varepsilon(x)|^2 - |x- m_\alpha^f(\rho)|^2\big|\le C(\Omega)$, and applying the Dominated Convergence Theorem, we have 
\begin{align}\label{a1}
\lim_{\varepsilon\to 0}\int_0^T\big\||h^\varepsilon(x)|^2 - \big|x- m_\alpha^f(\rho)\big|^2\big\|^2_{L^{\frac{\gamma}{\gamma-1}}(\Omega)}dt =0.   
\end{align}
And taking $r>0$ enough large such that $\Omega\subset B_r(0)$ and using (\ref{nrhos}), we get
\begin{align}\label{a2}
\|\nabla\varphi\cdot\nabla\rho_\varepsilon\|_{L^2(0,T;L^{\gamma}(\Omega))}\leq C(\Omega,T).
\end{align}
Hence combining \eqref{a1} and \eqref{a2}, it holds that
\begin{align}\label{JJ1}
J_1=\||h^\varepsilon(x)|^2 - |x- m_\alpha^f(\rho)|^2\|_{L^2(0,T;L^{\frac{\gamma}{\gamma-1}}(\Omega))}\|\nabla\varphi\cdot\nabla\rho_\varepsilon\|_{L^2(0,T;L^{\gamma}(\Omega))}\to 0\quad {\mbox as}~~\varepsilon\to 0.
\end{align}
For $J_2$, using the weak convergence relation (\ref{Rntovar2}), it follows that
\begin{align}\label{JJ2}
J_2=\Big|\int_0^T\int_{\Omega}|x- m_\alpha^f(\rho)|^2\nabla(\rho_{\varepsilon} - \rho)\cdot\nabla\varphi  dx dt\Big|\to 0 \quad {\mbox as}~~\varepsilon\to 0.
\end{align}
Consequently, substituting \eqref{JJ1} and \eqref{JJ2} into \eqref{JJ}, we have
\begin{align}\label{ii}
&\Big|\int_0^T\int_{\Omega}\big(|h^\varepsilon(x)|^2\nabla\rho_{\varepsilon} - |x -m_\alpha^f(\rho)|^2\nabla\rho\big)\cdot \nabla\varphi dxdt\Big|
\to 0 \qquad \mbox{as} \ \varepsilon\to 0.    
\end{align}

(iii) We confirm
\begin{align} 
&\Big|\int_0^T\int_{\Omega}\Big(\nabla (|h^\varepsilon|^2)\rho_{\varepsilon} - \nabla (\big|(x - m_\alpha^f(\rho))\big|^2)\rho\Big)\cdot\nabla\varphi dxdt\Big|\nonumber\\
\le & \Big|\int_0^T\int_{\Omega} \nabla\Big( |h^\varepsilon|^2 - \big|(x - m_\alpha^f(\rho))\big|^2\Big) \cdot\nabla\varphi \rho_{\varepsilon} dxdt\Big|\nonumber\\
&+\Big|\int_0^T\int_{\Omega} \nabla(\big|(x - m_\alpha^f(\rho))\big|^2)(\rho_{\varepsilon} - \rho)\cdot\nabla\varphi  dxdt\Big|\nonumber\\
\le & \|\nabla(|h^\varepsilon(x)|^2) - \nabla(|x- m_\alpha^f(\rho)|)^2\|_{L^2(0,T;L^2(\Omega))}\|\nabla\varphi\rho_\varepsilon
\|_{L^2(0,T;L^2(\Omega))}\nonumber\\
&+\| \nabla(\big|(x - m_\alpha^f(\rho))\big|^2)\cdot\nabla\varphi\|_{L^{(q^*)'}(0,T;L^{q'}(\Omega))}\|\rho_{\varepsilon} - \rho\|_{L^{q^*}(0,T;L^{q}(\Omega))},
\end{align}
where $1/{q'} + 1/q = 1$ and $1/{(q^*)'} + 1/{q^*} = 1$. 
Here, since for any $(t,x)\in[0,T]\times\Omega$, the sequence $\nabla(|h^\varepsilon|^2)$ convergence almost everywhere to $\nabla(|x-m_\alpha^f(\rho)|^2)$ in \eqref{gh2-gm2} and $\big|\nabla(|h^\varepsilon(x)|^2) - \nabla(|x- m_\alpha^f(\rho)|^2)\big|\le C(\Omega)$, applying the Dominated Convergence Theorem, we have 
\begin{align*}
\lim_{\varepsilon\to 0}\int_0^T\int_{\Omega}\big|\nabla|h^\varepsilon(x)|^2 - \nabla|x- m_\alpha^f(\rho)|^2\big|^2dxdt =0.
\end{align*}
Thus, we obtain 
\begin{align}\label{iii}
\Big|\int_0^T\int_{\Omega}\Big(\nabla (|h^\varepsilon|^2)\rho_{\varepsilon} - \nabla (\big|(x - m_\alpha^f(\rho))\big|^2)\rho\Big)\cdot\nabla\varphi dxdt\Big|  \to 0\quad \mathrm{as}\ \varepsilon\to0.  
\end{align}
   
\end{proof*}

Therefore, the four convergence relations hold in \eqref{i0},\eqref{i}, \eqref{ii} and \eqref{iii}.
For the above mentioned test function $\varphi\in\mathcal{C}_c^\infty((0,T)\times\R^d)$, we deduce $\rho$ satisfies the following equality of integration
\begin{align*}
&\int_0^T\int_{\R^d}\partial_t\rho \varphi+  \lambda\big(\big(x-m_\alpha^f(\rho)\big)\rho\big)\cdot\nabla\varphi+\frac{\sigma^2}{2}\nabla\big(\big|x-m_\alpha^f(\rho)\big|^2 \rho\big)\cdot\nabla\varphi dxdt=0.
\end{align*}
That gives $\rho$ is a weak solution of the equation \eqref{pde}.

One can further show that the solutions obtained above are actually unique.
\begin{prop}\textup{(Uniqueness)}
Let $f$ satisfy Assumption \ref{ass}, the non-negative  initial data $\rho_0\in L^1\cap L^2(\R^d)\cap \mathcal{P}_4(\R^d)$. Then, the weak solution $\rho$ to \eqref{pde} is unique.
\end{prop}
\begin{proof}
Let $\rho$ be any weak solution to the nonlinear PDE \eqref{pde}, and we consider the following linear SDE
        \begin{align}\label{linearSDE}
   d{X}_t = -\lambda\big({X}_t - m_\alpha^f(\rho_t)\big)dt + \sigma |{X}_t - m_\alpha^f(\rho_t)|dB_t\,.
\end{align}
Then the standard SDE theory tells us that \eqref{linearSDE} has a unique strong solution $X$, and one can define $\rho_t^X:=\mbox{Law }(X_t)$. Applying It\^{o}'s formula it implies that $\rho^X$ is a weak solution to the following linear PDE
\begin{equation}\label{linearPDE}
    \partial_t\rho^X= \lambda\nabla\cdot\big((x-m_\alpha^f(\rho_t))\rho^X\big) + \frac{\sigma^2}{2}\Delta\big(|x-m_\alpha^f(\rho_t)|^2\rho^X\big)\,.
\end{equation} 
Moreover, the solution to \eqref{linearPDE} is unique and we check it in the following. Let $\rho_1^X$ and $\rho_2^X$ be two  different solutions to \eqref{linearPDE} with the same initial data. Then $\overline{\rho^X}=\rho_1^X-\rho_2^X$ solves the problem
\begin{align}\label{pde1-2}
\begin{cases}
    &\partial_t\overline{\rho^X} = \lambda\nabla\cdot\big((x-m_\alpha^f(\rho_t))\overline{\rho^X}\big)+ \frac{\sigma^2}{2}\Delta\big(|x-m_\alpha^f(\rho_t)|^2\overline{\rho^X}\big),\\
&\overline{\rho^X} (0,x)=0.
\end{cases}
\end{align}
We multiply $\overline{\rho^X}$ on both sides of the equation \eqref{pde1-2}  and integrate over $\R^d$. Then, for all $0<t<T$, we have
\begin{align*}
\frac{1}{2}\frac{d}{dt}\int_{\R^d} |\overline{\rho^X}|^2dx 
=& -\lambda\int_{\R^d}\nabla\overline{\rho^X}\cdot\big(x-m_\alpha^f(\rho_t)\big)\overline{\rho^X}dx-\frac{\sigma^2}{2}\int_{\R^d}\nabla\overline{\rho^X}\cdot\nabla\big(\big|x-m_\alpha^f(\rho_t)\big|^2\overline{\rho^X}\big)dx\nonumber\\
=&-(\lambda + \sigma^2)\int_{\R^d}\nabla\overline{\rho^X}\cdot\big(x-m_\alpha^f(\rho_t)\big)\overline{\rho^X}dx-\frac{\sigma^2}{2}\int_{\R^d}|\nabla\overline{\rho^X}|^2\big|x-m_\alpha^f(\rho_t)\big|^2dx\nonumber\\
=&\frac{d}{2}(\lambda + \sigma^2)\int_{\R^d}|\overline{\rho^X}|^2dx-\frac{\sigma^2}{2}\int_{\R^d}|\nabla\overline{\rho^X}|^2\big|x-m_\alpha^f(\rho_t)\big|^2dx.
\end{align*}
Gr\"onwall's inequality then implies that for a fixed $T$ and any $t\in[0,T]$, 
$$\|\overline{\rho^X}\|_{L^2(\R^d)} \le \exp{\{\frac{d}{2}(\lambda + \sigma^2)T\}}\|\overline{\rho^X}(0,x)\|_{L^2(\R^d)} = 0.$$
Thus, the \eqref{linearPDE} has a unique solution. By assumption we know that $\rho$ is also a weak solution to the above linear PDE.
This means we must have that $\rho=\rho^X$. Thus $(X,\rho)=(X,\rho^X)$ is a solution to the nonlinear SDE 
\begin{align}\label{nonlinear}
   d\overline{X}_t = -\lambda\big(\overline{X}_t - m_\alpha^f(\rho_t)\big)dt + \sigma |\overline{X}_t - m_\alpha^f(\rho_t)|dB_t,\quad\mbox{ Law }(\OX_t)=\rho_t.
\end{align}
In summary, we have shown that any weak solution to the nonlinear PDE \eqref{pde} is the distribution of the nonlinear SDE \eqref{nonlinear}. By \cite[Theorem 3.2]{carrillo2018analytical}, the nonlinear SDE \eqref{nonlinear} has a unique strong solution. This concludes that the nonlinear PDE \eqref{pde} also has a unique solution.
\end{proof}
\section{The boundedness and $H^2$-regularity of solutions}
In this section, we prove Theorem \ref{thm2}, which primarily addresses the regularity of solutions. It enables us to obtain solutions with more favorable properties.
We first establish the $L^{\infty}$-estimate for the solution $\rho$ to the equation \eqref{pde}.
\begin{prop}\label{lem4.2}\textup{($L^\infty$-boundedness)}
Assume that $f$ satisfies Assumption \ref{ass}, and let the non-negative 
 initial data $\rho_0\in L^1\cap L^\infty(\R^d)\cap \mathcal{P}_4(\RR^d)$ and $\rho$ be the solution to the problem \eqref{pde}. Then, for any fixed $T>0$ , $\rho$ satisfies 
\begin{align}\label{rho}
\rho\in L^{\infty}(0,T; L^p(\R^d)), \qquad p\in[2,\infty].
\end{align}
\end{prop}
\begin{proof}
For any $2\le p<\infty$, multiply $\rho^{p-1}$ to the both hand in the regularization equation \eqref{pde} to get that for any fixed $T>0$, $0\leq t\leq T$, it holds that
\begin{align*}
\frac{1}{p}\frac{d}{dt}\|\rho\|_{L^p(\R^d)}^p\nonumber
=&-\lambda\int_{\R^d}\nabla(\rho^{p-1})\cdot (x -m_\alpha^f(\rho))\rho dx - \frac{\sigma^2}{2}\int_{\R^d}\nabla(\rho^{p-1})\cdot\nabla\big(|x -m_\alpha^f(\rho) |^2\rho\big)dx\nonumber\\
=& -\frac{\lambda(p-1)}{p}\int_{\R^d}\nabla(\rho^{p})\cdot (x -m_\alpha^f(\rho))dx \nonumber\\
&- \frac{\sigma^2}{2}\int_{\R^d}\nabla(\rho^{p-1})\cdot\big(2(x -m_\alpha^f(\rho))\rho + |x -m_\alpha^f(\rho) |^2\nabla\rho\big)dx \nonumber\\
=& \frac{p-1}{p} d(\lambda + \sigma^2)\int_{\R^d}|\rho|^p dx 
-
\frac{2\sigma^2(p-1)}{p^2}
\int_{\R^d}|\nabla\rho^\frac{p}{2}|^2
|x -m_\alpha^f(\rho)|^2dx\nonumber\\
\le& 
\frac{p-1}{p} d(\lambda + \sigma^2)~\|\rho\|_{L^p(\R^d)}^p .
\end{align*}
The Gr\"onwall's inequality implies that $\rho$ satisfies 
\begin{align*}
\|\rho\|_{L^p(\R^d)}
\le& 
\exp{\Big(\frac{p-1}{p} d(\lambda + \sigma^2)t\Big)}~\|\rho_0\|_{L^p(\R^d)}\\
\leq& \exp{\Big(\frac{p-1}{p} d(\lambda + \sigma^2)t\Big)}~\|\rho_0\|_{L^1(\R^d)}\|\rho_0\|^{\frac{p-1}{p}}_{L^\infty(\R^d)},\quad 2\leq p<\infty.
\end{align*}
Therefore, for any $t\in[0,T]$ we have
\begin{align*}
\|\rho\|_{L^\infty(\R^d)} =\lim_{p\to\infty}\|\rho\|_{L^p(\R^d)}
\le \exp{\big(d(\lambda + \sigma^2)t\big)}~\|\rho_0\|_{L^1(\R^d)}\|\rho_0\|_{L^\infty(\R^d)}.
\end{align*}
The demonstration of this lemma is hereby concluded.
\end{proof}

Next, we proceed to demonstrate the improved $H^2$-regularity property.
\begin{prop}\textup{($H^2$-regularity)}
Let $f$ satisfy Assumption \ref{ass}, the non-negative  initial data $\rho_0 \in L^1 \cap H^2(\R^d)\cap \mathcal{P}_4(\RR^d)$. Then the solution $\rho$ to problem \eqref{pde} satisfies
\begin{align}\label{h2reg}
\rho\in L^\infty(0,T;H^2(\R^d)),\quad   
\partial_t\rho\in L^\infty(0,T;L^2_{loc}(\mathbb{R}^d)).
\end{align}
\end{prop}
\begin{proof}
First, we apply the operator $D$ to both sides of \eqref{pde}, multiply the resulting equation by $D\rho$, and integrate the domain to get for any $t\in[0,T]$, it holds that
\begin{align}\label{J1}
\frac{1}{2}\frac{d}{dt}\int_{\R^d}|D\rho|^2 dx 
=&\int_{\R^d}D\rho\cdot D\Big[\lambda\nabla\cdot\big((x-m_\alpha^f(\rho))\rho\big) +\frac{\sigma^2}{2}\Delta\big(|x-m_\alpha^f(\rho)|^2\rho\big)\Big]dx\nonumber\nonumber\\
=&(\lambda +\sigma^2)d\int_{\R^d}|D\rho|^2 dx + (\lambda +2\sigma^2)\int_{\R^d} D\rho\cdot D\big( (x-m_\alpha^f(\rho))\cdot\nabla\rho\big)dx\nonumber\\
&+ \frac{\sigma^2}{2}\int_{\R^d} D\rho\cdot D\big( |x-m_\alpha^f(\rho)|^2\Delta\rho\big)dx\nonumber\\
=:&(\lambda +\sigma^2)d\int_{\R^d}|D\rho|^2 dx + J_1^1 +J_1^2.
\end{align}
For the sake of convenience, we expand $J_1^1$ in its component form.
So when $i,j=1,\cdots, d$, we check
\begin{align*}
\int_{\R^d} D\rho\cdot D\big( (x-m_\alpha^f(\rho))\cdot\nabla\rho\big)dx
=& \int_{\R^d}\sum_{j=1}^d\rho_{x_j}\Big(\sum_{i=1}^d\big((x_i - (m_\alpha^f(\rho))_i)\rho_{x_i}\big)\Big)_{x_j}dx\nonumber\\
=&\int_{\R^d}\sum_{j=1}^d\rho_{x_j}\sum_{i=1}^d\big(\delta_{ij}\rho_{x_i} + (x_i - (m_\alpha^f(\rho))_i)\rho_{x_ix_j}\big)dx.
\end{align*}
Thus, $J_1^1$ can be expressed as
\begin{align}\label{J11}
J_1^1=&(\lambda +2\sigma^2)\Big(\int_{\R^d}|D\rho|^2dx + \frac{1}{2}\int_{\R^d}\big(x - m_\alpha^f(\rho)\big)\cdot\nabla|D\rho|^2dx\Big)\nonumber\\
=&(1-\frac{d}{2})(\lambda +2\sigma^2)\int_{\R^d}|D\rho|^2dx.
\end{align}
For $J_1^2$, using integration by parts, we have
\begin{align}\label{J12}
J_1^2 =& \frac{\sigma^2}{2}\int_{\R^d} \nabla\rho\cdot \nabla\big( |x-m_\alpha^f(\rho)|^2\Delta\rho\big)dx\nonumber\\
=& -\frac{\sigma^2}{2}\int_{\R^d}  |x-m_\alpha^f(\rho)|^2|\Delta\rho|^2dx.
\end{align}
Substitute \eqref{J11} and \eqref{J12} into \eqref{J1}, we can verify
\begin{align}\label{D2}
\frac{1}{2}\frac{d}{dt}\int_{\R^d}|D\rho|^2 dx =& (1-\frac{d}{2})(\lambda +2\sigma^2)\int_{\R^d}|D\rho|^2dx - \frac{\sigma^2}{2}\int_{\R^d}|x -m_\alpha^f(\rho)|^2|\Delta\rho|^2dx \nonumber \\
\le& C(\lambda,\sigma,d)\int_{\R^d}|D\rho|^2dx.
\end{align}
Applying Gr\"onwall's inequality to \eqref{D2}, we obtain the following result 
\begin{align}\label{Drho}
\|D\rho\|_{L^\infty(0,T;L^{2}(\mathbb{R}^d))}\le C.
\end{align}

Second, applying the operator $D^2$ to both sides of \eqref{pde}, multiplying the resulting equation by $D^2\rho$, and integrating it in $\R^d$, we have for any $t\in[0,T]$, it holds that
\begin{align}\label{J2}
\frac{1}{2}\frac{d}{dt}\int_{\R^d}|D^2\rho|^2 dx 
=&\int_{\R^d}D^2\rho\cdot D^2\Big[\lambda\nabla\cdot\big((x-m_\alpha^f(\rho))\rho\big) +\frac{\sigma^2}{2}\Delta\big(|x-m_\alpha^f(\rho)|^2\rho\big)\Big]dx\nonumber\nonumber\\
=&(\lambda +\sigma^2)d\int_{\R^d}|D^2\rho|^2 dx + (\lambda +2\sigma^2)\int_{\R^d} D^2\rho\cdot D^2\big( (x-m_\alpha^f(\rho))\cdot\nabla\rho\big)dx\nonumber\\
&+ \frac{\sigma^2}{2}\int_{\R^d} D^2\rho\cdot D^2\big( |x-m_\alpha^f(\rho)|^2\Delta\rho\big)dx\nonumber\\
=:&(\lambda +\sigma^2)d\int_{\R^d}|D^2\rho|^2 dx + J_2^1 +J_2^2.
\end{align}
Here $"\cdot"$ denotes the Frobenius inner product of the Hessian matrices. In the following, we compute by using component forms of $J_2^1$ and $J_2^2$. For $J_2^1$, we consider
\begin{align*}
D^2\rho\cdot D^2\big( (x-m_\alpha^f(\rho))\cdot\nabla\rho\big)
=&\sum_{i,j=1}^d\rho_{x_ix_j}\Big(\sum_{k=1}^d(x_k -(m_\alpha^f(\rho))_k)\rho_{x_k}\Big)_{x_ix_j}\\
=&\sum_{i,j=1}^d\rho_{x_ix_j}\sum_{k=1}^d\Big(\delta_{ki}\rho_{x_k} + (x_k -(m_\alpha^f(\rho))_k)\rho_{x_kx_i}\Big)_{x_j}\\
=&\sum_{i,j=1}^d\rho_{x_ix_j}\sum_{k=1}^d\Big(\rho_{x_ix_j} + \delta_{kj}\rho_{x_kx_i} + (x_k -(m_\alpha^f(\rho))_k)\rho_{x_kx_ix_j}\Big)\\
=&2\sum_{i,j=1}^d(\rho_{x_ix_j})^2 +\frac{1}{2}\sum_{i,j,k=1}^d(x_k - (m_\alpha^f(\rho))_k)(\rho_{x_ix_j}^2)_{x_k}.
\end{align*}
Hence, $J_2^1$ may be written 
\begin{align}\label{J21}
J_2^1 =&(\lambda +2\sigma^2)\int_{\R^d}2|D^2\rho|^2 +\frac{1}{2}(x-m_\alpha^f(\rho))\cdot\nabla|D^2\rho|^2dx\nonumber\\
=&(2-\frac{d}{2})(\lambda +2\sigma^2)\int_{\R^d}|D^2\rho|^2dx.
\end{align}
For the component of $J_2^2$, we have 
\begin{align}\label{J221}
&D^2\rho\cdot D^2\big( |x-m_\alpha^f(\rho)|^2\Delta\rho\big)\nonumber\\
=&\sum_{i,j=1}^d\rho_{x_ix_j}\Big(|x -m_\alpha^f(\rho)|^2\Delta\rho\Big)_{x_ix_j}\nonumber\\
=&\sum_{i,j=1}^d\rho_{x_ix_j}\Big(2\delta_{ij}\Delta\rho + 4(x_i -(m_\alpha^f(\rho))_i)(\Delta\rho)_{x_j} + |x -m_\alpha^f(\rho)|^2(\Delta\rho)_{x_ix_j}\Big)\nonumber\\
=&2|\Delta\rho|^2 + \sum_{i,j=1}^d\Big(4\rho_{x_ix_j}(x_i -(m_\alpha^f(\rho))_i)(\Delta\rho)_{x_j} +\rho_{x_ix_j}|x -m_\alpha^f(\rho)|^2(\Delta\rho)_{x_ix_j}\Big).
\end{align}
Note that 
\begin{align}
&\int_{\R^d}\sum_{i,j=1}^d\Big(4\rho_{x_ix_j}(x_i -(m_\alpha^f(\rho))_i)(\Delta\rho)_{x_j} +\rho_{x_ix_j}|x -m_\alpha^f(\rho)|^2(\Delta\rho)_{x_ix_j}\Big)dx\nonumber\\
=
&2\int_{\R^d}\sum_{i,j=1}^d\rho_{x_ix_j}(x_i -(m_\alpha^f(\rho))_i)(\Delta\rho)_{x_j} dx
-\int_{\R^d}\sum_{i,j=1}^d\rho_{x_ix_ix_j}|x -m_\alpha^f(\rho)|^2(\Delta\rho)_{x_j}dx\nonumber\\
=& -2\int_{\R^d}\sum_{i,j=1}^d\Big(\rho_{x_ix_jx_j}(x_i -(m_\alpha^f(\rho))_i) - \delta_{ij}\rho_{x_ix_j}\Big)\Delta\rho dx\nonumber\\
&-\int_{\R^d}\sum_{i,j=1}^d\rho_{x_ix_ix_j}|x -m_\alpha^f(\rho)|^2(\Delta\rho)_{x_j}dx.
\end{align}
Consequently, by integrating both sides of the \eqref{J221}, we can express the $J_2^2$ as
\begin{align}\label{J22}
J_2^2
=&\frac{\sigma^2}{2} \int_{\R^d}2|\Delta\rho|^2 - \nabla|\Delta
\rho|^2\cdot(x-m_\alpha^f(\rho)) +2 |\Delta\rho|^2dx\nonumber\\
&-\frac{\sigma^2}{2} \int_{\R^d}|\nabla(\Delta\rho)|^2|x-m_\alpha^f(\rho)|^2 dx \nonumber\\
= & \frac{(4+d)\sigma^2}{2}\int_{\R^d}|D^2\rho|^2 dx-\frac{\sigma^2}{2} \int_{\R^d}|\nabla(\Delta\rho)|^2|x-m_\alpha^f(\rho)|^2 dx.
\end{align}
Substituting \eqref{J21} and \eqref{J22} into \eqref{J2} yields
\begin{align}\label{D22}
\frac{d}{dt}\int_{\R^d}|D^2\rho|^2 dx  =& C(\lambda,\sigma,d)\int_{\R^d}|D^2\rho|^2 dx 
-\sigma^2\int_{\R^d}|x -m_\alpha^f(\rho)|^2|\nabla(\Delta\rho)|^2dx\nonumber\\
\le &C(\lambda,\sigma,d)\int_{\R^d}|D^2\rho|^2 dx.
\end{align}
Again applying Gr\"onwall's inequality to \eqref{D22},  we derive the following result
\begin{align}\label{Drho2}
\|D^2\rho\|_{L^\infty(0,T;L^{2}(\mathbb{R}^d))}\le C.
\end{align}

Next, we consider the derivative with respect to time. Since $\rho$ is the weak solution, for any test function $\varphi(x)\in C_c^\infty(\R^d)$, $\Omega:=\mathrm{ssup}\varphi$ and $t\in[0,T]$, we deduce
\begin{align}\label{ptrho1}
|<\partial_t\rho,\varphi>|
\le& \lambda\Big|\int_{\Omega}\nabla\varphi\cdot \big((x-m_\alpha^f(\rho))\rho\big) dx\Big|
+\frac{\sigma^2}{2}\Big|\int_{\Omega} \nabla\varphi\cdot\nabla(|x-m_\alpha^f(\rho)|^2 \rho) dx\Big|\nonumber\\
\le& \lambda\Big|\int_{\Omega}\varphi\nabla\cdot \big((x-m_\alpha^f(\rho))\rho\big) dx\Big|
+\frac{\sigma^2}{2}\Big|\int_{\Omega} \varphi\Delta(|x-m_\alpha^f(\rho)|^2 \rho) dx\Big|\nonumber\\
\le& d(\lambda + \sigma^2)\Big|\int_{\Omega}\varphi\rho dx\Big| + (\lambda + 2\sigma^2)\Big|\int_{\Omega}\varphi (x-m_\alpha^f(\rho))\cdot\nabla\rho dx\Big|\nonumber\\
& + \frac{\sigma^2}{2}\Big|\int_{\Omega} \varphi|x-m_\alpha^f(\rho)|^2 \Delta\rho dx\Big|\nonumber\\
\le& C(\Omega)\Big(\|\rho\|_{L^2(\R^d)} + \|\nabla\rho\|_{L^2(\Omega)} + \|\Delta\rho\|_{L^2(\Omega)}\Big)\|\varphi\|_{L^2(\Omega)}.
\end{align}
Hence, we have  
\begin{align}\label{Drhot}
&\|\partial_t\rho\|_{L^\infty(0,T;L^2_{loc}(\mathbb{R}^d))}\le C.
\end{align}

Therefore, by \eqref{Drho}, \eqref{Drho2}
and \eqref{Drhot}, we obtain that the \eqref{h2reg} holds true. These estimates of solution $\rho$ guarantee the existence of strong solutions to the model \eqref{pde}.
\end{proof}

\bibliographystyle{abbrv}
\bibliography{refCBO}
\end{document}